\documentclass[british, polish, 12pt]{article}
\usepackage[a4paper, margin=2cm]{geometry}
\usepackage[T1]{fontenc}
\usepackage[latin9]{inputenc}
\usepackage{amsmath}
\usepackage{amssymb}
\usepackage{esint}
\usepackage{url}
\usepackage{babel}

\makeatletter

\newtheorem{theo}{Theorem}

\newtheorem{prop}{Proposition}  
\newtheorem{coro}{Corollary}

\newtheorem{rem}{Remark}
\newtheorem{lema}{Lemma}

 \global\long\def\sbr#1{\left[ #1\right] }
 \global\long\def\cbr#1{\left\{  #1\right\}  }
 \global\long\def\rbr#1{\left(#1\right)}

 \global\long\def\R{\mathbb{R}}

 \global\long\def\dd#1{\textnormal{d}#1}

 \global\long\def\ab{[a,b]}
 \global\long\def\ra{\rightarrow}
 \global\long\def\TTV#1#2#3{\text{TV}^{#3}\!\rbr{#1,#2}}
\global\long\def\TTVemph#1#2#3{\emph{TV}^{#3}\!\rbr{#1,#2}}
 \global\long\def\V#1#2#3{\text{V}^{#3}\!\rbr{#1,#2}}

 \global\long\def\ns{\infty}
 \global\long\def\f{:\left[a,b\right]\ra\R}
 \global\long\def\TV{\text{TV}}

\newenvironment{proof}{\noindent{\bf Proof.}}{\rightline{$\blacksquare$}}

\title{A~new inequality for the Riemann-Stieltjes integrals driven by irregular signals in Banach spaces}

\author{Rafa{\l{}} M. \L{}ochowski}

\begin{document}

\selectlanguage{british}%

\maketitle

\abstract{We prove an inequality of the Lo\'{e}ve-Young type for the  Riemann-Stieltjes integrals driven by irregular signals attaining their values in Banach spaces and, as a result, we derive a new theorem on the existence of the Riemann-Stieltjes integrals driven by such signals. Also, for any $p\ge1$ we introduce the space of regulated signals $f:\ab\ra W$ ($a<b$ are real numbers and $W$ is a Banach space), which may be uniformly approximated with
accuracy $\delta>0$ by signals whose total variation is of order
$\delta^{1-p}$ as $\delta\ra0+$ and prove that they satisfy the assumptions of the theorem. Finally, we derive more exact, rate-independent characterisations of the irregularity of the integrals driven by such signals.}
\smallskip


\section{Introduction}

The first aim of this paper is a generalisation of the results of \cite{LochowskiGhomrasniMMAS:2015}
and \cite{LochowskiJIA:2015} to the functions attaining
their values not only in $\R$ but in more general spaces. Next, to
obtain more precise results, for any $p\ge1$ we introduce the space ${\cal U}^{p}\left(\left[a,b\right], W\right)$
of regulated functions/signals $f:\ab\ra W$ ($a<b$ are real numbers and
$W$ is a Banach space), which may be uniformly approximated with
accuracy $\delta>0$ by functions whose total variation is of order
$\delta^{1-p}$ as $\delta\ra0+.$ This way we will obtain a result about the existence of the Riemann-Stieltjes integral $\int_a^b f \dd g$ for functions from ${\cal U}^{p}\left(\left[a,b\right]\right)$ and ${\cal U}^{q}\left(\left[a,b\right]\right)$ whenever $p,q>1,$ $p^{-1}+q{-1}>1.$ 
Results of this type were earlier obtained by Young \cite{Young:1936}, \cite{Young:1938} and D'ya\v{c}kov \cite{Dyackov:1988} (for very detailed account see \cite[Chapt. 3]{NorvaisaConcrete:2010}) but they were  expressed in terms of $p$- or (more general) $\phi$-variations. 

In \cite{LochowskiGhomrasniMMAS:2015} the following variational problem
was considered: given real $a<b,$ $c>0,$ a regulated function/signal
$f\f$ (for the definition of a regulated function see the next section)
and $x\in\left[f\left(a\right)-c/2,f\left(a\right)+c/2\right],$ find
the infimum of total variations of all functions $f^{c,x}\f$ which uniformly
approximate $f$ with accuracy $c/2,$ 
\[
\left\Vert f-f^{c,x}\right\Vert _{\ab,\ns}:=\sup_{a\le t\le b}\left|f\left(t\right)-f^{c,x}\left(t\right)\right|\le c/2,
\]
and start from $x,$ $f^{c,x}\left(a\right)=x.$ Recall that for
$g\f$ its total variation is defined as 
\[
\TTV g{\ab}{}:=\sup_{n}\sup_{a\le t_{0}<t_{1}<\ldots<t_{n}\le b}\sum_{i=1}^{n}\left|g\left(t_{i}\right)-g\left(t_{i-1}\right)\right|.
\]
This infimum is well approximated by the \emph{truncated variation}
of $f,$ defined as 
\begin{equation}
\TTV f{\ab}c:=\sup_{n}\sup_{a\le t_{0}<t_{1}<\ldots<t_{n}\le b}\sum_{i=1}^{n}\max\left\{ \left|f\left(t_{i}\right)-f\left(t_{i-1}\right)\right|-c,0\right\} ,\label{eq:tv_def}
\end{equation}
and the following bounds hold 
\[
\TTV f{\ab}c\le\inf_{f^{c,x}\in B\left(f,c/2\right),f^{c,x}\left(a\right)=x}\TTV{f^{c,x}}{\ab}{}\le\TTV f{\ab}c+c,
\]
where $B\left(f,c/2\right):=\left\{ g:\left\Vert f-g\right\Vert _{\ab,\ns}\le c/2\right\} $
(see \cite[Theorem 4 and Remark 15]{LochowskiGhomrasniMMAS:2015}).
Moreover, we have 
\begin{equation}
\inf_{f^{c}\in B\left(f,c/2\right)}\TTV{f^{c}}{\ab}{}=\TTV f{\ab}c\label{eq:tv_infimum}
\end{equation}
Unfortunately, this result is no more valid for functions attaining
their values in more general metric spaces. 
\begin{rem} \label{koles} It is not difficult
to see that (\ref{eq:tv_infimum}) does not hold even for $f$ attaining
its values in $\R^{2}$ with with $|\cdot|$ understood as the Euclidean
norm in $\R^{2}.$ Indeed, let $f:\left[0,2\right]\ra\R^{2}$ be defined
with the formula $f\left(t\right)=\left(\cos\left(2\pi\left\lfloor t\right\rfloor /3\right),\sin\left(2\pi\left\lfloor t\right\rfloor /3\right)\right).$
We have $\TTVemph f{\left[0,2\right]}{\sqrt{3}}=0,$ but there exist no
sequence of functions $f_{n}:\left[0,2\right]\ra\R^{2},$ $n=1,2,\ldots,$
such that $\left\Vert f-f_{n}\right\Vert _{[0,2],\ns}\le\sqrt{3}/2$
and $\lim_{n\ra+\ns}\TTVemph {f_{n}}{\left[0,2\right]}{}=0.$ Thus $\inf_{f^{c}\in B\left(f,c/2\right)}\TTVemph{f^{c}}{\ab}{}>\TTVemph f{\ab}c.$
\end{rem} 
Remark \ref{koles} answers (negatively) the question posed few years ago by Krzysztof Oleszkiewicz, if the truncated variation is the greatest lower bound for the total variation of functions from $B\rbr{f,c/2}$ attaining values in $\R^d,$ $d=2,3,\ldots$ or in other spaces than $\R.$
Fortunately, it is possible to state an easy estimate of the
left side of (\ref{eq:tv_infimum}) in terms of the truncated variation
of $f,$ for $f$ attaining its values in \emph{any} metric space
(to define the total variation and the truncated variation of $f$
attaining its values in the metric space $\left(E,d\right)$ we just
replace $|f\left(t_{i}\right)-f\left(t_{i-1}\right)|$ by the distance
$d\left(f\left(t_{i}\right),f\left(t_{i-1}\right)\right)$); see Theorem
\ref{Fact_multi_dim}.

One of the applications of Theorem \ref{Fact_multi_dim} will be the
generalisation of the results of \cite{LochowskiJIA:2015}
on the existence of the Riemann-Stieltjes integral. We will consider
the case when the integrand and the integrator attain their values
in Banach spaces. The restriction to the Banach spaces stems from
the fact that the method of our proof requires multiple application
of summation by parts and proceeding to the limit of a Cauchy sequence, which may be done in a straightforward way in any Banach space. This way we will obtain
a general theorem on the existence of the Riemann-Stieltjes integral
along a path in some Banach space $\left(E,\left\Vert \cdot\right\Vert _{E}\right)$
(with the integrand being a path in the space $L\left(E,V\right)$
of continuous linear mappings $F:E\ra V,$ where $V$ is another
Banach space) as well as an improved version of the Lo\'{e}ve-Young inequality
for integrals driven by irregular paths in this space.

The famous Lo\'{e}ve-Young inequality may be stated as follows. If $f:\ab\ra L\left(E,V\right)$
and $g:\ab\ra E$ are two regulated functions with no common points
of discontinuity and $f$ and $g$ have finite $p$- and $q$-variations
respectively, where $p > 1,$ $q > 1$ and $p^{-1}+q^{-1}>1,$ then the
Riemann-Stieltjes integral $\int_{a}^{b}f\dd g$ exists and one has
the following estimate 
\begin{equation}
\left\Vert \int_{a}^{b}f\dd g-f\left(a\right)\left[g\left(b\right)-g\left(a\right)\right]\right\Vert \le\tilde{C}_{p,q}\left(V^{p}\left(f,[a,b]\right)\right)^{1/p}\left(V^{q}\left(g,[a,b]\right)\right)^{1/q}.\label{eq:LYinequality}
\end{equation}
Here 
\[
V^{p}\left(f,[a,b]\right):=\sup_{n}\sup_{a\le t_{0}<t_{1}<\ldots<t_{n}\le b}\sum_{i=1}^{n}\left\Vert f\left(t_{i}\right)-f\left(t_{i-1}\right)\right\Vert _{L\left(E,V\right)}^{p},
\]

\[
V^{q}\left(g,[a,b]\right):=\sup_{n}\sup_{a\le t_{0}<t_{1}<\ldots<t_{n}\le b}\sum_{i=1}^{n}\left\Vert g\left(t_{i}\right)-g\left(t_{i-1}\right)\right\Vert _{E}^{q}
\]
denote $p$- and $q$-variation of $f$ and $g$ respectively (sometimes called the {\em strong} variation). The
original Lo\'{e}ve-Young estimate, with the constant $\tilde{C}_{p,q}=1+\zeta\left(1/p+1/q\right),$
where $\zeta$ is the famous Riemann zeta function, was formulated
for real functions in \cite{Young:1936}. The counterpart of this
inequality for more general, Banach space-valued functions, with the
constant $\tilde{C}_{p,q}=4^{1/p+1/q}\zeta\left(1/p+1/q\right),$
is formulated in the proof of \cite[Theorem 1.16]{LyonsCaruana:2007}.
Our, improved version of (\ref{eq:LYinequality}) is the following
\begin{align*}
& \left\Vert \int_{a}^{b}f\dd g-f\left(a\right)\left[g\left(b\right)-g\left(a\right)\right]\right\Vert \\
& \le C_{p,q}\left(V^{p}\left(f,[a,b]\right)\right)^{1-1/q}\left\Vert f\right\Vert _{\text{osc},\ab}^{1+p/q-p}\left(V^{q}\left(g,[a,b]\right)\right)^{1/q},
\end{align*}
where $\left\Vert f\right\Vert _{\text{osc},\ab}:=\sup_{a\le s<t\le b}\left\Vert f\left(s\right)-f\left(t\right)\right\Vert _{L\left(E,V\right)}$ and $C_{p,q}$ is a universal constant depending on $p$ and $q$ only.
Notice that always \[
\left(V^{p}\left(f,[a,b]\right)\right)^{1/p-\left(1-1/q\right)}\ge\left\Vert f\right\Vert _{\text{osc},\ab}^{1+p/q-p}.
\]

These results may be applied for example 
when $f$ and $g$ are trajectories of $\alpha$-stable processes $X^1,$ $X^2$ with $\alpha \in (1,2).$ However, since the obtained results are formulated
in terms of rate-independent functionals, like the truncated variation
or $p$-variation, they remain valid when $f\left(t\right)=F\left(X^1\left(A\left(t\right)\right)\right)$
and $g\left(t\right)=G\left(X^2\left(B\left(t\right)\right)\right)$ 
(with the technical assumption that the jumps of $f$ and $g$ do not occur at the same time) 
where $A,B:\left[0,+\ns\right)\ra\left[0,+\ns\right)$ are piecewise monotonic, possibly random, changes of time (i.e. there exist $0=T_{0}<T_{1}<\ldots$
such that $T_{n}\ra+\ns$ almost surely as $n\ra+\ns$ and $A$ and $B$
are monotonic on each interval $\left(T_{i-1},T_{i}\right),$ $i=1,2,\ldots$), while $F,G:\R\ra\R$ are locally Lipschitz. 

It appears that it is possible to derive weaker conditions under which
the improved Lo\'{e}ve-Young inequality still holds, and we will prove
that it still holds (and the Riemann-Stieltjes integral $\int_{a}^{b}f\dd g$
exists) for functions $f$ and $g$ with no commont poins of discontinuity,
satisfying 
\[\sup_{\delta>0}\delta^{p-1}\TTV f{\ab}{\delta}<\ns \text{ and }
\sup_{\delta>0}\delta^{q-1}\TTV g{\ab}{\delta}<\ns
\] respectively. 
Moreover, in such a case the indefinite integral $ I\left(t\right):=\int_{a}^{t}f\dd g$
reveals similar irregularity as the integrator $g,$ namely, $\sup_{\delta>0}\delta^{q-1}\TTV I{\ab}{\delta}<\ns.$ 
We will also prove that for any $p \ge 1$ the 
class of functions
$f:\ab\ra W,$ where $W$ is some Banach space, such that $\TTV f{\ab}{\delta}=O\left(\delta^{1-p}\right)$
as $\delta\ra0+,$ is a Banach space. We will denote it by ${\cal U}^{p}\rbr{[a,b],W}.$ The property $f \in {\cal U}^{p}\rbr{[a,b],W}$ is weaker than the existence of  $p$-variation  but stronger than the existence of $q$-variation for some $q>p.$

From early work of Lyons \cite{Lyons:1994} it is well known that whenever $f$ and $g$ have finite $p$- and $q$-variations
respectively,  $p > 1,$ $q > 1$ and $p^{-1}+q^{-1}>1,$ then 
the indefinite integral $I(\cdot)$ has finite $q$-variation. However, it is also
well known that a symmetric $\alpha$-stable process $X$  with $\alpha \in [1,2]$ has finite $p$-variation for any $p>\alpha$ while
its $\alpha$-variation is infinite (on any proper, compact subinterval
of $\left[0,+\ns\right)$), see for example \cite[Theorem 4.1]{BlumenthalGetoor:1960}. 
Thus, if for example $f\left(t\right)=F\left(X^1\left(A\left(t\right)\right)\right)$
and $g\left(t\right)=G\left(X^2\left(B\left(t\right)\right)\right)$
are like in a former paragraph, we can say that $I\left(\cdot\right)$
has finite $p$-variation, on any compact subinterval
of $\left[0,+\ns\right)$ for any $p>\alpha,$ but can not say much more.
From our results it will follow that $I\left(\cdot\right)\in{\cal U}^{\alpha}\rbr{[0,t],\R}$ for any $t\ge 0.$ As far as we know, no such result is known in the case when the integrator has finite $\phi$-variation except the already mentioned case $\phi(x) = |x|^q.$

Let us comment on the organization of the paper. In the next section we prove very general estimates for $\inf_{g \in B\rbr{f,c/2}} \TTV g{[a,b]}{},$ for regulated $f:\ab \ra E,$ where $E$ is any metric space, in terms of the truncated variation of $f.$ Next, in the third section, we use obtained estimates to prove a new theorem on the existence of the Riemann-Stieltjes integral driven by irregular paths in Banach spaces. In the proofs we follow closely \cite{LochowskiJIA:2015}. In  the last, fourth section we introduce the Banach spaces ${\cal U}^{p}\rbr{[a,b],W},$ $p\ge 1,$ (subsection 4.1) and in subsection 4.2 obtain more exact estimates of the rate-independent irregularity of functions from these spaces (in terms of $\phi-$variation). In the last subsection we deal with the irregularity of the integrals driven by signals from the spaces ${\cal U}^{p}\rbr{[a,b],W},$ $p\ge 1.$

\section{Estimates for the variational problem}

Let $\left(E,d\right)$ be a metric space with the metric $d$. For
given reals $a<b$ we say that the function $f:\ab\ra E$ is \emph{regulated}
if it has right limits $f\left(t+\right)$ for any $t\in\left[a,b\right)$
and left limits $f\left(t-\right)$ for any $t\in\left(a,b\right].$
If $E$ is complete then a necessary and sufficient condition for
$f$ to be regulated is that it is an uniform limit of step functions
(see \cite[Theorem 7.6.1]{Dieudonne:1969}).

Let $f:\left[a,b\right]\ra E$ be regulated. For $c>0$ let us consider
the family $B\left(f,c/2\right)$ of \emph{all} functions $g:\ab\ra E$
such that $\sup_{t\in\ab}d\left(f\left(t\right),g\left(t\right)\right)\le c/2.$
We will be interested in the followng variational problem: find 
\begin{equation}
\inf_{g\in B\left(f,c/2\right)}\TTV g{\left[a,b\right]}{},\label{eq:infimum}
\end{equation}
where 
\[
\TTV g{\ab}{}:=\sup_{n}\sup_{a\le t_{0}<t_{1}<\ldots<t_{n}\le b}\sum_{i=1}^{n}d\left(g\left(t_{i}\right),g\left(t_{i-1}\right)\right).
\]

To state our first main result let us define 
\[
\TTV g{\ab}c:=\sup_{n}\sup_{a\le t_{0}<t_{1}<\ldots<t_{n}\le b}\sum_{i=1}^{n}\max\left\{ d\left(g\left(t_{i}\right),g\left(t_{i-1}\right)\right)-c,0\right\} .
\]

\begin{theo} \label{Fact_multi_dim} For any regulated {$f:\left[a,b\right]\ra E$
}there exists a step function $f^{c}:\ab\ra E$ such that $\sup_{t\in\ab}d\left(f\left(t\right),f^{c}\left(t\right)\right)\le c/2$
and for any $\lambda>1,$ $\TTVemph{f^{c}}{\left[a,b\right]}{}\le\lambda\cdot\TTVemph f{\ab}{\left(\lambda-1\right)c/\left(2\lambda\right)}.$
Thus the following estimates hold 
\[
\TTVemph f{\ab}c\le\inf_{g\in B\left(f,c/2\right)}\TTVemph g{\left[a,b\right]}{}\le\inf_{\lambda>1}\lambda\cdot\TTVemph f{\ab}{\left(\lambda-1\right)c/\left(2\lambda\right)}.
\]
In particular, taking $\lambda=2$ we get the double-sided estimate
\[
\TTVemph f{\ab}c\le\inf_{g\in B\left(f,c/2\right)}\TTVemph g{\left[a,b\right]}{}\le2\cdot\TTVemph f{\ab}{c/4}.
\]

Moreover, if $E$ is a vector normed space with the norm $\left\Vert \cdot\right\Vert _{E}$
then there exists $f^{c,lin}:\ab\ra E$ such that $f^{c,lin}$ is
piecewise linear, jumps of $f^{c,lin}$ occur only at the points where
the jumps of $f$ occur, $\sup_{t\in\ab}\left\Vert f\left(t\right)-f^{c,lin}\left(t\right)\right\Vert _{E}\le c$
and $\TTVemph{f^{c,lin}}{\left[a,b\right]}{}=\TTVemph{f^{c}}{\left[a,b\right]}{}.$
\end{theo} \begin{proof} The estimate from below 
\[
\inf_{g\in B\left(f,c/2\right)}\TTV g{\left[a,b\right]}{}\ge\TTV f{\ab}c
\]
follows immediately from the triangle inequality: if $\sup_{t\in\ab}d\left(f\left(t\right),g\left(t\right)\right)\le c/2$
then for any $a\le s<t\le b,$ 
\begin{align*}
\max\left\{ d\left(g\left(t\right),g\left(s\right)\right)-c,0\right\}  & \le\max\left\{ d\left(g\left(t\right),g\left(s\right)\right)-d\left(g\left(t\right),f\left(t\right)\right)-d\left(g\left(s\right),f\left(s\right)\right),0\right\} \\
 & \le d\left(f\left(t\right),f\left(s\right)\right).
\end{align*}

The estimate from above follows from the following greedy algorithm.
Let {us consider the sequence of times defined
in the following way:} $\tau_{0}=a$ and {for
$n=1,2,\ldots$ 
\[
\tau_{n}=\begin{cases}  
\!\begin{aligned}
 & \inf\left\{ t\in\left(\tau_{n-1},b\right]:d\left(f\left(t\right),f\left(\tau_{n-1}\right)\right)>c/2\right\}   \\
 &  \quad \quad \quad \quad  \text{if } {\tau_{n-1}<b} \text{ and } d\left(f\left(\tau_{n-1}\right),f\left(\tau_{n-1}+\right)\right)<c/2;
\end{aligned}
\\
\!\begin{aligned}
&  \inf\left\{ t\in\left(\tau_{n-1},b\right]:d\left(f\left(t\right),f\left(\tau_{n-1}+\right)\right)>c/2\right\}  \\ & \quad \quad \quad \quad \text{if } {\tau_{n-1}<b} \text{ and } d\left(f\left(\tau_{n-1}\right),f\left(\tau_{n-1}+\right)\right)\ge c/2;
\end{aligned} \\
\!\begin{aligned}
& +\ns \quad \quad \text{otherwise.}
\end{aligned}
\end{cases}
\]
Note that, since $f$ is regulated, $\lim_{n\ra+\ns}\tau_{n}=+\ns.$
(We apply the convention that $\inf\emptyset=+\ns$.) Now we define
a step function $f^{c}\in B\left(f,c/2\right)$ in the following way.
For each $n=1,2,\ldots$ such that $\tau_{n-1}<b$ we put} 
\begin{itemize}
\item if {$d\left(f\left(\tau_{n-1}\right),f\left(\tau_{n-1}+\right)\right)<c/2$
then } 
\[
f^{c}\left(t\right):=f\left(\tau_{n-1}\right)\mbox{ for }t\in\left[\tau_{n-1},\tau_{n}\right)\cap\left[a,b\right];
\]

\item if {$d\left(f\left(\tau_{n-1}\right),f\left(\tau_{n-1}+\right)\right)\ge c/2$
then $f^{c}\left(\tau_{n-1}\right):=f\left(\tau_{n-1}\right)$ and
\[
f^{c}\left(t\right):=f\left(\tau_{n-1}+\right)\mbox{ for }t\in\left(\tau_{n-1},\tau_{n}\right)\cap\left[a,b\right].
\]
} 
\end{itemize}
This way the function $f^{c}$ is defined for all $t\in\ab.$

It is not difficult to see that the just constructed $f^{c}$ satisfies
$\sup_{t\in\ab}d\left(f\left(t\right),f^{c}\left(t\right)\right)\le c/2$
and for each $n=1,2,\ldots$ such that $\tau_{n}\le b$ we have 
\begin{itemize}
\item if $d\left(f\left(\tau_{n-1}\right),f\left(\tau_{n-1}+\right)\right)<c/2$
then 
\begin{equation}
d\left(f^{c}\left(\tau_{n-1}\right),f^{c}\left(\tau_{n-1}+\right)\right)=0\label{eq:i_raz}
\end{equation}
and 
\end{itemize}
\begin{equation}
d\left(f^{c}\left(\tau_{n-1}+\right),f^{c}\left(\tau_{n}\right)\right)=d\left(f\left(\tau_{n-1}\right),f\left(\tau_{n}\right)\right)\ge c/2;\label{eq:i_dwa}
\end{equation}

\begin{itemize}
\item if $d\left(f\left(\tau_{n-1}\right),f\left(\tau_{n-1}+\right)\right)\ge c/2$
then 
\begin{equation}
d\left(f^{c}\left(\tau_{n-1}\right),f^{c}\left(\tau_{n-1}+\right)\right)=d\left(f\left(\tau_{n-1}+\right),f\left(\tau_{n-1}+\right)\right)\ge c/2\label{eq:i_trzy}
\end{equation}
and 
\begin{equation}
d\left(f^{c}\left(\tau_{n-1}+\right),f^{c}\left(\tau_{n}\right)\right)=d\left(f\left(\tau_{n-1}+\right),f\left(\tau_{n}\right)\right)\ge c/2.\label{eq:i_cztery}
\end{equation}

\end{itemize}
Let $N=\max\left\{ n:\tau_{n-1}<b\right\} .$ From an elementary inequality $x\le\lambda\max\left\{ x-\frac{\lambda-1}{2\lambda}c,0\right\} $
valid for any $x\in\left\{ 0\right\} \cup\left[c/2,+\ns\right)$ and
$\lambda>1,$   and (\ref{eq:i_raz})
- (\ref{eq:i_cztery})  we have 
\begin{align*}
\TTV{f^{c}}{\left[a,b\right]}{} & =\sum_{n=1}^{N}\left\{ d\left(f^{c}\left(\tau_{n-1}\right),f^{c}\left(\tau_{n-1}+\right)\right)+d\left(f^{c}\left(\tau_{n-1}+\right),f^{c}\left(\tau_{n}\wedge b\right)\right)\right\} \\
 & \le\lambda\sum_{n=1}^{N}\max\left\{ d\left(f^{c}\left(\tau_{n-1}\right),f^{c}\left(\tau_{n-1}+\right)\right)-\frac{\lambda-1}{2\lambda}c,0\right\} \\
 & \quad+\lambda\sum_{n=1}^{N}\max\left\{ d\left(f^{c}\left(\tau_{n-1}\right),f^{c}\left(\tau_{n}\wedge b\right)\right)-\frac{\lambda-1}{2\lambda}c,0\right\} \\
 & \le\lambda\sum_{n=1}^{N}\max\left\{ d\left(f\left(\tau_{n-1}\right),f\left(\tau_{n-1}+\right)\right)-\frac{\lambda-1}{2\lambda}c,0\right\} \\
 & \quad+\lambda\sum_{n=1}^{N}\max\left\{ d\left(f\left(\tau_{n-1}+\right),f\left(\tau_{n}\wedge b\right)\right)-\frac{\lambda-1}{2\lambda}c,0\right\} \\
 & \le\lambda\TTV f{\ab}{\left(\lambda-1\right)c/\left(2\lambda\right)}.
\end{align*}
Thus, since $f^{c}\in B\left(f,c/2\right)$ and $\lambda$ was an
arbitrary number from the interval $\left(1,+\ns\right),$ we have
\[
\inf_{g\in B\left(f,c/2\right)}\TTV g{\left[a,b\right]}{}\le\TTV{f^{c}}{\left[a,b\right]}{}\le\inf_{\lambda>1}\lambda\cdot\TTV f{\ab}{\left(\lambda-1\right)c/\left(2\lambda\right)}.
\]

The construction of the function $f^{c,lin}$ is similar. For $\tau_{n},$
$n=0,1,\ldots,$ such that $\tau_{n}\le b,$ we define $f^{c,lin}\left(\tau_{n}\right)=f\left(\tau_{n}\right)$
and for $t\in\left(\tau_{n-1},\tau_{n}\right)\cap\ab,$ $n=0,1,\ldots$
such that $\tau_{n-1}<b$ it is defined in the following way. 
\begin{itemize}
\item If {$d\left(f\left(\tau_{n-1}\right),f\left(\tau_{n-1}+\right)\right)<c/2,$
$\tau_{n}\le b$ and $f\left(\tau_{n}-\right)\neq f\left(\tau_{n}\right)$
then for $t\in\left(\tau_{n-1},\tau_{n}\right)$ we put 
\[
f^{c,lin}\left(t\right):=f\left(\tau_{n-1}\right);
\]
} 
\item if $d\left(f\left(\tau_{n-1}\right),f\left(\tau_{n-1}+\right)\right)<c/2$
and $\tau_{n}=+\ns$ then for $t\in\left(\tau_{n-1},b\right]$ we
put 
\[
f^{c,lin}\left(t\right):=f\left(\tau_{n-1}\right);
\]

\item if {$d\left(f\left(\tau_{n-1}\right),f\left(\tau_{n-1}+\right)\right)<c/2,$
$\tau_{n}\le b$ and $f\left(\tau_{n}-\right)=f\left(\tau_{n}\right)$
then for $t\in\left(\tau_{n-1},\tau_{n}\right)$ we put 
\[
f^{c,lin}\left(t\right):=\frac{\tau_{n}-t}{\tau_{n}-\tau_{n-1}}f\left(\tau_{n-1}\right)+\frac{t-\tau_{n-1}}{\tau_{n}-\tau_{n-1}}f\left(\tau_{n}\right);
\]
} 
\item if {$d\left(f\left(\tau_{n-1}\right),f\left(\tau_{n-1}+\right)\right)\ge c/2,$
$\tau_{n}\le b$ and $f\left(\tau_{n}-\right)\neq f\left(\tau_{n}\right)$
then for $t\in\left(\tau_{n-1},\tau_{n}\right)$ we put 
\[
f^{c,lin}\left(t\right):=f\left(\tau_{n-1}+\right);
\]
} 
\item if {$d\left(f\left(\tau_{n-1}\right),f\left(\tau_{n-1}+\right)\right)\ge c/2$
and $\tau_{n}=+\ns$ then for $t\in\left(\tau_{n-1},b\right]$ we
put 
\[
f^{c,lin}\left(t\right):=f\left(\tau_{n-1}+\right);
\]
} 
\item if $d\left(f\left(\tau_{n-1}\right),f\left(\tau_{n-1}+\right)\right)\ge c/2,$
$\tau_{n}\le b$ and $f\left(\tau_{n}-\right)=f\left(\tau_{n}\right)$
then for $t\in\left(\tau_{n-1},\tau_{n}\right)$ 
\[
f^{c,lin}\left(t\right):=\frac{\tau_{n}-t}{\tau_{n}-\tau_{n-1}}f\left(\tau_{n-1}+\right)+\frac{t-\tau_{n-1}}{\tau_{n}-\tau_{n-1}}f\left(\tau_{n}\right).
\]

\end{itemize}
It is straightforward to verify that $\sup_{t\in\ab}\left\Vert f\left(t\right)-f^{c,lin}\left(t\right)\right\Vert _{E}\le c,$
$\TTV{f^{c,lin}}{\left[a,b\right]}{}=\TTV{f^{c}}{\left[a,b\right]}{}$ and the jumps of $f^{c,lin}$ occur only at the points where
the jumps of $f$ occur.
\end{proof}

\section{Integration of irregular signals in Banach spaces}

\selectlanguage{british}%
Directly from the definition it follows that the truncated variation
is a superadditive functional of the interval, i.e. for $c\geq0$
and any $d\in\left(a,b\right)$ 
\begin{equation}
\TTV f{[a,b]}{\delta}\geq\TTV f{\left[a;d\right]}{\delta}+\TTV f{\left[d,b\right]}{\delta}.\label{eq:superadditivity}
\end{equation}
Moreover, if $\left(E,\left\Vert \cdot\right\Vert _{E}\right)$ is
a normed vector space (with the norm $\left\Vert \cdot\right\Vert _{E}$)
we also have the following easy estimate of the truncated variation
of a function $g:\ab\ra E$ perturbed by some other function $h:\ab\ra E:$
\begin{equation}
\TTV{g+h}{[a,b]}{\delta}\leq\TTV g{[a,b]}{\delta}+\TTV h{[a,b]}0,\label{eq:TV_variation}
\end{equation}
which stems directly from the inequality: for $a\leq s<t\leq b,$
\begin{align*}
& \max\left\{ \left\Vert g\left(t\right)+h\left(t\right)-\left\{ g\left(s\right)+h\left(s\right)\right\} \right\Vert _{E}-\delta,0\right\} \\ 
& \quad \quad \quad \quad \quad \quad \le \max\left\{ \left\Vert g\left(t\right)-g\left(s\right)\right\Vert _{E}-\delta,0\right\}  +\left\Vert h\left(t\right)-h\left(s\right)\right\Vert _{E}.
\end{align*}

Let now $\left(E,\left\Vert \cdot\right\Vert _{E}\right),$ $\left(W,\left\Vert \cdot\right\Vert _{W}\right)$
be Banach spaces, $\left(V,\left\Vert \cdot\right\Vert _{V}\right)$
be another Banach space and $\left(L\left(E,V\right),\left\Vert \cdot\right\Vert _{L\left(E,V\right)}\right)$
be the space of continuous linear mappings $F:E\ra V$ with the
norm $\left\Vert F\right\Vert _{L\left(E,V\right)}=\sup_{e\in E:\left\Vert e\right\Vert _{E}=1}\left\Vert F\cdot e\right\Vert _{V}.$
Throughout the rest of this paper we will assume that $f:\ab\ra W$
and $g:\ab\ra E.$ We will often encounter the situation when $W=L\left(E,V\right).$

Relations (\ref{eq:superadditivity}) and (\ref{eq:TV_variation}),
together with Theorem \ref{Fact_multi_dim}, will allow us to establish
the following result.

\begin{theo} \label{main} Let $f:[a,b]\rightarrow L\left(E,V\right)$
and $g:[a,b]\rightarrow E$ be two regulated functions
which have no common points of discontinuity. Let $\eta_{0}\geq\eta_{1}\geq\ldots$
and $\theta_{0}\geq\theta_{1}\geq\ldots$ be two sequences of positive
numbers, such that $\eta_{k}\downarrow0,$ $\theta_{k}\downarrow0$
as $k\rightarrow+\infty.$ Define $\eta_{-1}:=\frac{1}{2}\sup_{a\leq t\leq b}\left\Vert f\left(t\right)-f\left(a\right)\right\Vert _{L\left(E,V\right)}$
and 
\[
S:=4\sum_{k=0}^{+\infty}3^{k}\eta_{k-1}\cdot\TTVemph g{[a,b]}{\theta_{k}/4}+4\sum_{k=0}^{\infty}3^{k}\theta_{k}\cdot\TTVemph f{[a,b]}{\eta_{k}/4}.
\]
If $S<+\infty$ then the Rieman-Stieltjes integral $\int_{a}^{b}f\mathrm{d}g$
exists and one has the following estimate 
\begin{equation}
\left\Vert \int_{a}^{b}f\mathrm{d}g-f\left(a\right)\left[g\left(b\right)-g\left(a\right)\right]\right\Vert _{V}\leq S.\label{eq:estimate_integral}
\end{equation}
\end{theo} \begin{rem} For $\xi,s,t\in\ab$ by $f\left(\xi\right)\left[g\left(t\right)-g\left(s\right)\right]$
we mean the value of the linear mapping $f\left(\xi\right)$ evaluated
at the vector $g\left(t\right)-g\left(s\right)$ (i.e. the element
of the space $V$) and the Riemann-Stieltjes integral $\int_{a}^{b}f\mathrm{d}g$
is understood as the limit (if it exists) of the sums 
\[
\sum_{i=1}^{n}f\left(\xi_{i}^{\left(n\right)}\right)\left[g\left(t_{i}^{\left(n\right)}\right)-g\left(t_{i-1}^{\left(n\right)}\right)\right],
\]
where for $n=1,2,\ldots,$ $i=1,2,\ldots,n,$ $a=t_{0}^{\left(n\right)}<t_{1}^{\left(n\right)}<\ldots<t_{n}^{\left(n\right)}=b,$
$\xi_{i}^{\left(n\right)}\in\left[t_{i-1}^{\left(n\right)},t_{i}^{\left(n\right)}\right]$
and $\lim_{n\ra+\ns}\max_{1\le i\le n}\left(t_{i}^{\left(n\right)}-t_{i-1}^{\left(n\right)}\right)=0.$
\end{rem} The proof of Theorem \ref{main} will be based on the following
lemmas.
\begin{lema}[summation by parts in a Banach space]\label{sum_parts} Let $f:\ab\rightarrow L\left(E,V\right),$
$g:\ab\rightarrow E$ and $c=t_{0}<t_{1}<\ldots<t_{n}=d$ be any partition
of the interval $\left[c,d\right]\subset[a,b].$ Let $\xi_{0}=c$
and $\xi_{1},\ldots,\xi_{n}$ be such that $t_{i-1}\leq\xi_{i}\leq t_{i}$
for $i=1,2,\ldots,n,$ then 
\begin{equation}
\sum_{i=1}^{n}\left[f\left(\xi_{i}\right)-f\left(c\right)\right]\left[g\left(t_{i}\right)-g\left(t_{i-1}\right)\right]=\sum_{i=1}^{n}\left[f\left(\xi_{i}\right)-f\left(\xi_{i-1}\right)\right]\left[g\left(d\right)-g\left(t_{i-1}\right)\right].\label{eq:sum_by_parts}
\end{equation}
\end{lema} 
\begin{proof} {For $i=1,2\ldots,n$ let us denote $f_{i}=f\left(\xi_{i}\right)-f\left(\xi_{i-1}\right),$
$g_{i}=g\left(t_{i}\right)-g\left(t_{i-1}\right).$ We have 
\begin{align*}
&\sum_{i=1}^{n}\left[f\left(\xi_{i}\right)-f\left(c\right)\right]\left[g\left(t_{i}\right)-g\left(t_{i-1}\right)\right]  =\sum_{i=1}^{n}\left(\sum_{j=1}^{i}f_{j}\right)g_{i}\\
&=\sum_{j=1}^{n}f_{j}\left(\sum_{i=j}^{n}g_{i}\right)=\sum_{i=1}^{n}\left[f\left(\xi_{i}\right)-f\left(\xi_{i-1}\right)\right]\left[g\left(d\right)-g\left(t_{i-1}\right)\right].
\end{align*}
} 
\end{proof}
\begin{lema} \label{lema_partition} Let $f:[a,b]\rightarrow L\left(E,V\right)$
and $g:[a,b]\rightarrow E$ be two regulated functions.
Let $c=t_{0}<t_{1}<\ldots<t_{n}=d$ be any partition of the interval
$\left[c,d\right]\subset[a,b]$ and let $\xi_{0}=c$ and
$\xi_{1},\ldots,\xi_{n}$ be such that $t_{i-1}\leq\xi_{i}\leq t_{i}$
for $i=1,2,\ldots,n.$ Then for $\delta_{-1}:=\frac{1}{2}\sup_{c\leq t\leq d}\left\Vert f\left(t\right)-f\left(c\right)\right\Vert _{L\left(E,V\right)}$
and any $\delta_{0}\geq\delta_{1}\geq\ldots\geq\delta_{r}>0$ and
$\varepsilon_{0}\geq\varepsilon_{1}\geq\ldots\geq\varepsilon_{r}>0$
the following estimate holds 
\begin{eqnarray*}
 &  & \left\Vert\sum_{i=1}^{n}f\left(\xi_{i}\right)\left[g\left(t_{i}\right)-g\left(t_{i-1}\right)
\right]-f\left(c\right)\left[g\left(d\right)-g\left(c\right)\right]\right\Vert_V\\
 &  & \leq 4\sum_{k=0}^{r}3^{k}\delta_{k-1}\cdot\TTVemph g{\left[c,d\right]}{\varepsilon_{k}/4}+4\sum_{k=0}^{r}3^{k}\varepsilon_{k}\cdot\TTVemph f{\left[c,d\right]}{\delta_{k}/4}+n\delta_{r}\varepsilon_{r}.
\end{eqnarray*}
\end{lema}
\begin{proof} The proof goes exactly along the same lines as the
proof of \cite[Lemma 1]{LochowskiJIA:2015} with the obvious
changes. The idea is to utilize Theorem \ref{Fact_multi_dim} and approximate
the functions $g$ an $f$ by two piecewise linear functions $g^{\varepsilon_{0}}:\ab\ra E$
and $f^{\delta_{0}}:\ab\ra L\left(E,V\right)$ satisfying the following
conditions: 
\begin{equation}
\sup_{t\in\left[c,d\right]}\left\Vert g\left(t\right)-g^{\varepsilon_{0}}\left(t\right)\right\Vert _{E}\leq\varepsilon_{0}\mbox{ and }\TTV{g^{\varepsilon_{0}}}{\left[c,d\right]}{}\le2\TTV g{\left[c,d\right]}{\varepsilon_{0}/4},\label{eq:jeden}
\end{equation}
and 
\begin{equation}
\sup_{t\in\left[c,d\right]}\left\Vert f\left(t\right)-f^{\delta_{0}}\left(t\right)\right\Vert _{L\left(E,V\right)}\leq\delta_{0}\mbox{ and }\TTV{f^{\delta_{0}}}{\left[c,d\right]}{}\le2\TTV f{\left[c,d\right]}{\delta_0/4}.\label{eq:dwa}
\end{equation}
We have 
\begin{align*}
 & \left\Vert \sum_{i=1}^{n}f\left(\xi_{i}\right)\left[g\left(t_{i}\right)-g\left(t_{i-1}\right)\right]-f\left(c\right)\left[g\left(d\right)-g\left(c\right)\right]\right\Vert _{V}\\
 & =\left\Vert \sum_{i=1}^{n}\left[f\left(\xi_{i}\right)-f\left(c\right)\right]\left[g\left(t_{i}\right)-g\left(t_{i-1}\right)\right]\right\Vert _{V}\\
 & \le\left\Vert \sum_{i=1}^{n}\left[f\left(\xi_{i}\right)-f\left(c\right)\right]\left[g^{\varepsilon_{0}}\left(t_{i}\right)-g^{\varepsilon_{0}}\left(t_{i-1}\right)\right]\right\Vert _{V}\\
 & \quad +\left\Vert \sum_{i=1}^{n}\left[f\left(\xi_{i}\right)-f\left(c\right)\right]\left[\left(g-g^{\varepsilon_{0}}\right)\left(t_{i}\right)-\left(g-g^{\varepsilon_{0}}\right)\left(t_{i-1}\right)\right]\right\Vert _{V}.
\end{align*}
By (\ref{eq:jeden}) we estimate the first summand 
\begin{align*}
\left\Vert \sum_{i=1}^{n}\left[f\left(\xi_{i}\right)-f\left(c\right)\right]\left[g^{\varepsilon_{0}}\left(t_{i}\right)-g^{\varepsilon_{0}}\left(t_{i-1}\right)\right]\right\Vert _{V} & \le2\delta_{-1}\cdot\TTV{g^{\varepsilon_{0}}}{\left[c,d\right]}{}\\
 & \le4\delta_{-1}\cdot\TTV g{\left[c,d\right]}{\varepsilon_{0}/4}.
\end{align*}
By the summation by parts and then by (\ref{eq:dwa}) we estimate
the second summand 
\begin{eqnarray}
 &  & \left\Vert \sum_{i=1}^{n}\left[f\left(\xi_{i}\right)-f\left(c\right)\right]\left[\left(g-g^{\varepsilon_{0}}\right)\left(t_{i}\right)-\left(g-g^{\varepsilon_{0}}\right)\left(t_{i-1}\right)\right]\right\Vert _{V}\nonumber \\
 &  & =\left\Vert \sum_{i=1}^{n}\left[f\left(\xi_{i}\right)-f\left(\xi_{i-1}\right)\right]\left[\left(g-g^{\varepsilon_{0}}\right)\left(d\right)-\left(g-g^{\varepsilon_{0}}\right)\left(t_{i-1}\right)\right]\right\Vert _{V} \nonumber \\
 &  & \le\left\Vert \sum_{i=1}^{n}\left[f^{\delta_{0}}\left(\xi_{i}\right)-f^{\delta_{0}}\left(\xi_{i-1}\right)\right]\left[\left(g-g^{\varepsilon_{0}}\right)\left(d\right)-\left(g-g^{\varepsilon_{0}}\right)\left(t_{i-1}\right)\right]\right\Vert _{V}\nonumber \\
 &  & \quad +\left\Vert \sum_{i=1}^{n}\left[\left(f-f^{\delta_{0}}\right)\left(\xi_{i}\right)-\left(f-f^{\delta_{0}}\right)\left(\xi_{i-1}\right)\right]\left[\left(g-g^{\varepsilon_{0}}\right)\left(d\right)-\left(g-g^{\varepsilon_{0}}\right)\left(t_{i-1}\right)\right]\right\Vert _{V}\nonumber \\
 &  & \le2\varepsilon_{0}\cdot\TTV{f^{\delta_{0}}}{\left[c,d\right]}{}+4n\delta_{0}\varepsilon_{0}\leq4\varepsilon_{0}\cdot\TTV f{\left[c,d\right]}{\delta_{0}/4}+4n\delta_{0}\varepsilon_{0}.\label{eq:jeden1}
\end{eqnarray}
Repeating these arguments, by induction we get 
\begin{eqnarray}
 &  & \left\Vert\sum_{i=1}^{n}f\left(\xi_{i}\right)\left[g\left(t_{i}\right)-g\left(t_{i-1}\right)\right]-f\left(c\right)\left[g\left(d\right)-g\left(c\right)\right]\right\Vert_V\nonumber \\
 &  & \leq4\sum_{k=0}^{r}\delta_{k-1}\cdot\TTV{g_{k}}{\left[c,d\right]}{\varepsilon_{k}/4}
+4\sum_{k=0}^{r}\varepsilon_{k}\cdot\TTV{f_{k}}{\left[c,d\right]}{\delta_{k}/4}+4n\delta_{r}\varepsilon_{r},\label{eq:induction}
\end{eqnarray}
where $g_{0}\equiv g,$ $f_{0}\equiv f$ and for $k=1,2,\ldots,r,$
$g_{k}:=g_{k-1}-g_{k-1}^{\varepsilon_{k-1}},$ $f_{k}:=f_{k-1}-f_{k-1}^{\delta_{k-1}}$
are defined similarly as $g_{1}$ and $f_{1}.$ Since $\varepsilon_{k}\leq\varepsilon_{k-1}$
for $k=1,2,\ldots,r,$ by (\ref{eq:TV_variation}) and the fact that
the function $\delta\mapsto\TTV h{\left[c,d\right]}{\delta}$ is non-increasing,
we estimate 
\begin{eqnarray*}
\TTV{g_{k}}{\left[c,d\right]}{\varepsilon_{k}/4} & = & \TTV{g_{k-1}-g_{k-1}^{\varepsilon_{k-1}}}{\left[c,d\right]}{\varepsilon_{k}/4}\\
 & \leq & \TTV{g_{k-1}}{\left[c,d\right]}{\varepsilon_{k}/4}+\TTV{g_{k-1}^{\varepsilon_{k-1}}}{\left[c,d\right]}0\\
 & \le & \TTV{g_{k-1}}{\left[c,d\right]}{\varepsilon_{k}/4}+2 \TTV{g_{k-1}}{\left[c,d\right]}{\varepsilon_{k-1}/4}\\
 & \leq & 3\TTV{g_{k-1}}{\left[c,d\right]}{\varepsilon_{k}/4}.
\end{eqnarray*}
Hence, by recursion, for $k=1,2,\ldots,r,$ 
\[
\TTV{g_{k}}{\left[c,d\right]}{\varepsilon_{k}/4}\leq3^{k}\TTV g{\left[c,d\right]}{\varepsilon_{k}/4}.
\]
Similarly, for $k=1,2,\ldots,r,$ we have 
\[
\TTV{f_{k}}{\left[c,d\right]}{\delta_{k}/4}\leq3^{k}\TTV f{\left[c,d\right]}{\delta_{k}/4}.
\]
By (\ref{eq:induction}) and last two estimates we get the desired
estimate.
\end{proof} 

Now we proceed to the proof of Theorem \ref{main}.

\begin{proof}
Again, the proof goes exactly along the same lines as the proof of
\cite[Theorem 1]{LochowskiJIA:2015} with the obvious changes.
Therefore we present here the main steps and for details we refer
to the proof of \cite[Theorem 1]{LochowskiJIA:2015}. It is
enough to prove that for any two partitions $\pi=\left\{ a=a_{0}<a_{1}<\ldots<a_{l}=b\right\} ,$
$\rho=\left\{ a=b_{0}<b_{1}<\ldots<b_{m}=b\right\} $ and $\nu_{i}\in\left[a_{i-1},a_{i}\right],$
$\xi_{j}\in\left[b_{j-1},b_{j}\right],$ $i=1,2,\ldots,l,$ $j=1,2,\ldots,m,$
the difference 
\[
\left\Vert \sum_{i=1}^{l}f\left(\nu_{i}\right)\left[g\left(a_{i}\right)-g\left(a_{i-1}\right)\right]-\sum_{j=1}^{m}f\left(\xi_{j}\right)\left[g\left(b_{j}\right)-g\left(b_{j-1}\right)\right]\right\Vert _{V}
\]
is as small as we please, provided that the meshes of the partitions
$\pi$ and $\rho$, defined as $\mbox{mesh}\left(\pi\right):=\max_{i=1,2,\ldots,l}\left(a_{i}-a_{i-1}\right)$
and $\mbox{mesh}\left(\rho\right):=\max_{j=1,2,\ldots,m}\left(b_{j}-b_{j-1}\right)$
respectively, are sufficiently small. Define 
\[
\sigma=\pi\cup\rho=\left\{ a=s_{0}<s_{1}<\ldots<s_{n}=b\right\} 
\]
and for $i=1,2,\ldots,l$ we estimate 
\begin{eqnarray*}
 &  & \left\Vert f\left(\nu_{i}\right)\left[g\left(a_{i}\right)-g\left(a_{i-1}\right)\right]-\sum_{k:s_{k-1},s_{k}\in\left[a_{i-1},a_{i}\right]}f\left(s_{k-1}\right)\left[g\left(s_{k}\right)-g\left(s_{k-1}\right)\right]\right\Vert _{V}\\
 &  & \leq\left\Vert f\left(\nu_{i}\right)\left[g\left(a_{i}\right)-g\left(a_{i-1}\right)\right]-f\left(a_{i-1}\right)\left[g\left(a_{i}\right)-g\left(a_{i-1}\right)\right]\right\Vert _{V}\\
 &  & \quad +\left\Vert \sum_{k:s_{k-1},s_{k}\in\left[a_{i-1},a_{i}\right]}f\left(s_{k-1}\right)\left[g\left(s_{k}\right)-g\left(s_{k-1}\right)\right]-f\left(a_{i-1}\right)\left[g\left(a_{i}\right)-g\left(a_{i-1}\right)\right]\right\Vert _{V}.
\end{eqnarray*}
Choose $N=1,2,\ldots.$ By the assumption that $f$ and $g$ have
no common points of discontinuity, if $\mbox{mesh}\left(\pi\right)$
is sufficiently small, for $i=1,2,\ldots,l$ we have 
\begin{equation}
\sup_{a_{i-1}\leq s\leq a_{i}}\left\Vert f\left(s\right)-f\left(a_{i-1}\right)\right\Vert _{L\left(E,V\right)}\leq\eta_{N-1}\label{eq:ind_f}
\end{equation}
or 
\begin{equation}
\sup_{a_{i-1}\leq s\leq a_{i}}\left\Vert g\left(a_{i}\right)-g\left(s\right)\right\Vert _{E}\leq\theta_{N-1}.\label{eq:ind_J}
\end{equation}
{ }Now, for $i=1,2,\ldots,l$ we define
\[
S_{i}:=4\sum_{j=0}^{+\infty}3^{j}\eta_{j-1}\cdot\TTV g{\left[a_{i-1},a_{i}\right]}{\theta_{j}/4}+4\sum_{j=0}^{+\infty}3^{j}\theta_{j}\cdot\TTV f{\left[a_{i-1},a_{i}\right]}{\eta_{j}/4}
\]
and for $i$ such that (\ref{eq:ind_f}) holds, we set $\delta_{-1}:=\frac{1}{2}\eta_{N-1},$
$\delta_{j}:=\eta_{N+j},$ $\varepsilon_{j}:=\theta_{N+j},$ $j=0,1,2,\ldots.$
By Lemma \ref{lema_partition} we estimate 
\begin{eqnarray*}
 &  & \left\Vert \sum_{k:s_{k-1},s_{k}\in\left[a_{i-1},a_{i}\right]}f\left(s_{k-1}\right)\left[g\left(s_{k}\right)-g\left(s_{k-1}\right)\right]-f\left(a_{i-1}\right)\left[g\left(a_{i}\right)-g\left(a_{i-1}\right)\right]\right\Vert _{V}\\
 &  & \leq 4 \sum_{j=0}^{+\infty}3^{j}\delta_{j-1}\cdot\TTV g{\left[a_{i-1},a_{i}\right]}{\varepsilon_{j}/4}+4\sum_{j=0}^{+\infty}3^{j}\varepsilon_{j}\cdot\TTV f{\left[a_{i-1},a_{i}\right]}{\delta_{j}/4}\\
 &  & \leq 4\sum_{j=0}^{+\infty}3^{j}\eta_{N+j-1}\cdot\TTV g{\left[a_{i-1},a_{i}\right]}{\theta_{N+j}/4}+4\sum_{j=0}^{+\infty}3^{j}\theta_{N+j}\cdot\TTV f{\left[a_{i-1},a_{i}\right]}{\eta_{N+j}/4}\\
 &  & \leq3^{-N}S_{i}.
\end{eqnarray*}
Similarly, 
\begin{eqnarray*}
\left\Vert f\left(\nu_{i}\right)\left[g\left(a_{i}\right)-g\left(a_{i-1}\right)\right]-f\left(a_{i-1}\right)\left[g\left(a_{i}\right)-g\left(a_{i-1}\right)\right]\right\Vert _{V} & \leq & 3^{-N}S_{i}.
\end{eqnarray*}
Hence 
\begin{equation}
\left\Vert f\left(\nu_{i}\right)\left[g\left(a_{i}\right)-g\left(a_{i-1}\right)\right]-\sum_{k:s_{k-1},s_{k}\in\left[a_{i-1},a_{i}\right]}f\left(s_{k-1}\right)\left[g\left(s_{k}\right)-g\left(s_{k-1}\right)\right]\right\Vert _{V} \leq 2\cdot 3^{-N}S_{i}.\label{eq:ineq_i}
\end{equation}
The truncated variation is a superadditive function of the interval,
from which we have 
\[
\sum_{i=1}^{l}\TTV g{\left[a_{i-1},a_{i}\right]}{\theta_{j}/4}\leq\TTV g{[a,b]}{\theta_{j}/4},
\]
\[
\sum_{i=1}^{l}\TTV f{\left[a_{i-1},a_{i}\right]}{\eta_{j}/4}\leq\TTV f{[a,b]}{\eta_{j}/4}.
\]

Let $I$ be the set of all indices, for which (\ref{eq:ind_f}) holds.
By (\ref{eq:ineq_i}) and last two inequalities, summing over $i\in I$
we get the estimate 
\begin{eqnarray}
 &  & \left\Vert \sum_{i\in I}\left\{ f\left(\nu_{i}\right)\left[g\left(a_{i}\right)-g\left(a_{i-1}\right)\right]-\sum_{k:s_{k-1},s_{k}\in\left[a_{i-1},a_{i}\right]}f\left(s_{k-1}\right)\left[g\left(s_{k}\right)-g\left(s_{k-1}\right)\right]\right\} \right\Vert _{V}\nonumber \\
 &  & \leq2 \cdot 3^{-N}\sum_{i\in I}S_{i}\leq2\cdot3^{-N}\sum_{i=1}^{l}S_{i}\leq2\cdot3^{-N}S.\label{eq:estim_I-1}
\end{eqnarray}
 Now, let $J$ be the set of all indices, for which (\ref{eq:ind_J})
holds. By the summation by parts (Lemma \ref{sum_parts}), by Lemma \ref{lema_partition} and
the superadditivity of the truncated variation we get 
\begin{eqnarray}
 &  & \left\Vert \sum_{i\in J}\left\{ f\left(\nu_{i}\right)\left[g\left(a_{i}\right)-g\left(a_{i-1}\right)\right]-\sum_{k:s_{k-1},s_{k}\in\left[a_{i-1},a_{i}\right]}f\left(s_{k-1}\right)\left[g\left(s_{k}\right)-g\left(s_{k-1}\right)\right]\right\} \right\Vert _{V}\nonumber \\
 &  & \leq4\cdot3^{-N}\sum_{i\in J}S_{i}\leq4\cdot3^{-N}\sum_{i=1}^{l}S_{i}\leq4\cdot3^{-N}S.\label{eq:estim_J-1}
\end{eqnarray}
Finally, from (\ref{eq:estim_I-1}) and (\ref{eq:estim_J-1}) we get
\[
\left\Vert \sum_{i=1}^l f\left(\nu_{i}\right)\left[g\left(a_{i}\right)-g\left(a_{i-1}\right)\right]-\sum_{k=1}^{n}f\left(s_{k-1}\right)\left[g\left(s_{k}\right)-g\left(s_{k-1}\right)\right]\right\Vert _{V}\leq6\cdot3^{-N}S.
\]
Similar estimate holds for 
\[
\left\Vert \sum_{j=1}^{m}f\left(\xi_{j}\right)\left[g\left(b_{j}\right)-g\left(b_{j-1}\right)\right]-\sum_{k=1}^{n}f\left(s_{k-1}\right)\left[g\left(s_{k}\right)-g\left(s_{k-1}\right)\right]\right\Vert _{V},
\]
provided that $\mbox{mesh}\left(\rho\right)$ is sufficiently small.
Hence 
\[
\left\Vert \sum_{i=1}^{l}f\left(\nu_{i}\right)\left[g\left(a_{i}\right)-g\left(a_{i-1}\right)\right]-\sum_{j=1}^{m}f\left(\xi_{j}\right)\left[g\left(b_{j}\right)-g\left(b_{j-1}\right)\right]\right\Vert _{V}\leq12\cdot3^{-N}S,
\]
provided that $\mbox{mesh}\left(\pi\right)+\mbox{mesh}\left(\rho\right)$
is sufficiently small. 

Since $N$ may be arbitrary large, we get the
convergence of the approximating sums to an universal limit, which
is the Riemann-Stieltjes integral. 

The estimate (\ref{eq:estimate_integral})
follows directly from the proved convergence of approximating sums
to the Riemann-Stieltjes integral and Lemma \ref{lema_partition}. 
\end{proof}
\subsection{An improved version of the Lo\'{e}ve-Young inequality}
Now we will obtain an improved version of the Lo\'{e}ve-Young inequality for
integrals driven by irregular signals attaining their values in Banach
spaces. Our main tool will be Theorem \ref{main} and the following
simple relation between the rate of growth of the truncated variation
and finiteness of $p-$variation. If $V^{p}\left(f,\ab\right)<+\infty$
for some $p\ge 1,$ then for every $\delta>0,$ 
\begin{equation}
\TTV f{\ab}{\delta}\leq V^{p}\left(f,\ab\right)\delta^{1-p}.\label{eq:p_variation}
\end{equation}
This result folows immediately from the elementary estimate: for any
$x \ge 0,$ 
\[
\delta^{p-1}\max\left\{ x-\delta,0\right\} \le\begin{cases}
0 & \mbox{if }x\le\delta\\
x^{p} & \mbox{if }x>\delta
\end{cases}\le x^{p}.
\]
Notice also that if $V^{p}\left(f,\ab\right)<+\infty$ for some $p>0$
then $f$ is regulated. For $p\ge1$ and a Banach space $W$ by {${\cal V}^{p}\rbr{\ab,W}$ 
we will denote the Banach space of all functions $f:\ab\ra W$ such
that $V^{p}\left(f,\ab\right)<+\infty.$ By $\left\Vert f\right\Vert _{p-\text{var},[a,b]}$
we will denote the semi-norm 
\[
\left\Vert f\right\Vert _{p-\text{var},[a,b]}:=\left(V^{p}\left(f,\ab\right)\right)^{1/p}.
\]
\begin{coro} \label{corol_Young}
Let $f:\ab\rightarrow L\left(E,V\right),$ $g:\ab\ra E$ be two functions
with no common points of discontinuity. If $f\in{\cal V}^{p}\rbr{[a,b],L\left(E,V\right)}$
and $g\in{\cal V}^{q}\rbr{[a,b],E},$ where $p > 1,$ $q > 1,$ $p^{-1}+q^{-1}>1,$
then the Riemann Stieltjes $\int_{a}^{b}f\mathrm{d}g$ exists. Moreover,
there exist a constant $C_{p,q},$ depending on $p$ and $q$ only,
such that 
\[
\left\Vert \int_{a}^{b}f\mathrm{d}g-f\left(a\right)\left[g\left(b\right)-g\left(a\right)\right]\right\Vert _{V}\leq C_{p,q}\left\Vert f\right\Vert _{p-\emph{var},[a,b]}^{p-p/q}\left\Vert f\right\Vert _{\emph{osc},[a,b]}^{1+p/q-p}\left\Vert g\right\Vert _{q-\emph{var},[a,b]}.
\]
\end{coro} 
\begin{proof} By Theorem \ref{main} it is enough to
prove that for some positive sequences $\eta_{0}\geq\eta_{1}\geq\ldots$
and $\theta_{0}\geq\theta_{1}\geq\ldots,$ such that $\eta_{k}\downarrow0,$
$\theta_{k}\downarrow0$ as $k\rightarrow+\infty,$ and $\eta_{-1}=\sup_{a\leq t\leq b}\left\Vert f\left(t\right)-f\left(a\right)\right\Vert _{L\left(E,V\right)},$
one has 
\begin{align*}
S: & =4\sum_{k=0}^{+\infty}3^{k}\eta_{k-1}\cdot\TTV g{[a,b]}{\theta_{k}/4}+4\sum_{k=0}^{+\infty}3^{k}\theta_{k}\cdot\TTV f{[a,b]}{\eta_{k}/4},\\
 & \leq C_{p,q}\left\Vert f\right\Vert _{p-\text{var},[a,b]}^{p-p/q}\left\Vert f\right\Vert _{\text{osc},[a,b]}^{1+p/q-p}\left\Vert g\right\Vert _{q-\text{var},[a,b]}.
\end{align*}
The proof will follow from the proper choice of the sequences $\left(\eta_{k}\right)$
and $\left(\theta_{k}\right).$ Choose 
\[
\alpha=\frac{\sqrt{(q-1)(p-1)}+1}{2},\quad\beta=\frac{1}{2}\sup_{a\leq t\leq b}\left\Vert f\left(t\right)-f\left(a\right)\right\Vert _{L\left(E,V\right)},
\]
\[
\gamma=\left(V^{q}\left(g,[a,b]\right)\right/V^{p}\left(f,[a,b]\right))^{1/q}\beta^{p/q}
\]
and for $k=0,1,\ldots,$ define 
\[
\eta_{k-1}=\beta\cdot3^{-\left(\alpha^{2}/\left[\left(q-1\right)\left(p-1\right)\right]\right)^{k}+1}
\]
and
\[
\theta_{k}=\gamma\cdot3^{-\left(\alpha^{2}/\left[\left(q-1\right)\left(p-1\right)\right]\right)^{k}\alpha/\left(q-1\right)}.
\]
By (\ref{eq:p_variation}), similarly as in the proof of \cite[Corollary 2]{LochowskiJIA:2015},
one estimates that 
\begin{eqnarray*}
S & = & 4\sum_{k=0}^{+\infty}3^{k}\eta_{k-1}\cdot\TTV g{[a,b]}{\theta_{k}/4}+4 \sum_{k=0}^{+\infty}3^{k}\theta_{k}\cdot\TTV f{[a,b]}{\eta_{k}/4}\\
 & \leq & \left(\sum_{k=0}^{+\infty}3^{k+1-\left(1-\alpha\right)\left(\alpha^{2}/\left[\left(q-1\right)\left(p-1\right)\right]\right)^{k}}\right)4^{q}V^{q}\left(g,[a,b]\right)\beta\gamma^{1-q}\\
 &  & +\left(\sum_{k=0}^{+\infty}3^{k+1-p-\alpha\left(1-\alpha\right)\left(\alpha^{2}/\left[\left(q-1\right)\left(p-1\right)\right]\right)^{k}/\left(q-1\right)}\right)4^{p}V^{p}\left(f,[a,b]\right)\beta^{1-p}\gamma.
\end{eqnarray*}
Since $p^{-1}+q^{-1}>1$ we have $\left(q-1\right)\left(p-1\right)<1,$ $\alpha<1$ and  $\alpha^{2}/\left[\left(q-1\right)\left(p-1\right)\right]>1.$ From this
we easily infer that $S<+\infty$ and that the integral
$\int_{a}^{b}f\mathrm{d}g$ exists. Moreover, denoting 
\begin{align*}
C_{p,q} & = 4^{q} \sum_{k=0}^{+\infty}3^{k+1-\left(1-\alpha\right)\left(\alpha^{2}/\left[\left(q-1\right)\left(p-1\right)\right]\right)^{k}} \\ & \quad + 4^{p}\sum_{k=0}^{+\infty}3^{k+1-p-\alpha\left(1-\alpha\right)\left(\alpha^{2}/\left[\left(q-1\right)\left(p-1\right)\right]\right)^{k}/\left(q-1\right)} 
\end{align*}
we get 
\begin{align*}
S & \leq C_{p,q}\left(V^{q}\left(g,[a,b]\right)\right)^{1/q}\left(V^{p}\left(f,[a,b]\right)\right)^{1-1/q}\beta^{1+p/q-p}\\
 & \leq C_{p,q}\left\Vert g\right\Vert _{q-\text{var},[a,b]}\left\Vert f\right\Vert _{p-\text{var},[a,b]}^{p-p/q}\left\Vert f\right\Vert _{\text{osc},[a,b]}^{1+p/q-p}.
\end{align*}
\end{proof}

\section{Spaces ${\cal U}^{p}\rbr{[a,b],W}$} 

\subsection{${\cal U}^{p}\rbr{[a,b],W}$ as a Banach space.} Let
$p\ge1$ and $W$ be a Banach space. In this subsection we will prove
that the family ${\cal U}^{p}\rbr{[a,b],W}$ of functions functions
$f:\ab\ra W,$ such that $\sup_{\delta>0}\delta^{p-1}\TTV f{\ab}{\delta}<+\ns$
is a Banach space, and the functional 
\[
\left\Vert \cdot\right\Vert _{p-\text{TV},[a,b]}:{\cal U}^{p}\rbr{[a,b]}\ra\left[0,+\ns\right)
\]
defined by 
\begin{equation}
\left\Vert f\right\Vert _{p-\text{TV},[a,b]}:=\left(\sup_{\delta>0}\delta^{p-1}\TTV f{\ab}{\delta}\right)^{1/p}\label{eq:TV_norm}
\end{equation}
is a semi-norm on this space (while the functional $\left\Vert f\right\Vert _{\text{TV},p,[a,b]}=\left\Vert f\left(a\right)\right\Vert _{W}+\left\Vert f\right\Vert _{p-\text{TV},[a,b]}$
is a norm). From (\ref{eq:p_variation}) it follows that
\begin{equation}\label{norm_comp}
\left\Vert f\right\Vert _{p-\text{TV},[a,b]} \le \left\Vert f\right\Vert _{p-\text{var},[a,b]}
\end{equation}
thus ${\cal V}^{p}\rbr{[a,b],W}\subset{\cal U}^{p}\rbr{[a,b],W}.$
It appears that this inclusion is strict. For example, if $0\le a<b$ then a real,
symmetric $\alpha$-stable process $X$ with $\alpha \in (1,2]$ has finite $p$-variation for $p>\alpha$ while (as it was already mentioned in the Introduction) its $\alpha$-variation is a.s. infinite (on any proper, compact subinterval of $\left[0,+\ns\right)$). On the other hand, trajectories of $X$ belong a.s. to  ${\cal U}^{\alpha}\rbr{[0,t],\R}$ for any $t\ge 0,$ see \cite{LochowskiMilosLevy:2014}.
For another example see \cite[Theorem 17]{TronelVladimirov:2000}.

From the results of the next subsection it will also follow that
\[{\cal U}^{p}\rbr{[a,b],W}\subset\bigcap_{q>p}{\cal V}^{q}\rbr{[a,b],W}
\]
but, again, this inclusion is strict. 
\begin{rem} For further
justifcation of the importance of the spaces ${\cal U}^{p}\rbr{[a,b],W},$
$p > 1,$ let us also notice that if $W=L\left(E,V\right),$ $f$
belongs to ${\cal U}^{q}\rbr{[a,b],W}$ and $g$ belongs to
${\cal U}^{q}\rbr{[a,b],E}$ for some $q > 1$ such that $p^{-1}+q^{-1}>1,$
and $f$ and $g$ have no common points of discontinuity, then the
integral $\int_{a}^{b}f\dd g$ still exist and we have the estimate
\[
\left\Vert \int_{a}^{b}f\mathrm{d}g-f\left(a\right)\left[g\left(b\right)-g\left(a\right)\right]\right\Vert _{V}\leq C_{p,q}\left\Vert f\right\Vert _{p-TV,[a,b]}^{p-p/q}\left\Vert f\right\Vert _{\emph{osc},[a,b]}^{1+p/q-p}\left\Vert g\right\Vert _{q-\emph{TV},[a,b]},
\]
with the same constant $C_{p,q}$ which appears in Corollary \ref{corol_Young}.
This follows from the fact that in the proof of Corollary \ref{corol_Young}
we were using only estimate (\ref{eq:p_variation}), which now may
be replaced by the estimate 
\begin{equation} \label{TV_estimate}
\TTVemph f{\ab}{\delta}\le\left\Vert f\right\Vert _{p-\emph{TV},[a,b]}^{p} \delta^{1-p}
\end{equation}
valid for any $\delta>0,$ stemming directly from the definition of the norm $\left\Vert \cdot \right\Vert_{p-\emph{TV},[a,b]}.$ \end{rem} 
\begin{prop} \label{Banach_space}
For any $p\geq1,$ the functional $\left\Vert \cdot\right\Vert _{p-\emph{TV},[a,b]}$
is a seminorm and the functional $\left\Vert \cdot\right\Vert _{\emph{TV},p,[a,b]}$
is a norm on ${\cal U}^{p}\rbr{[a,b],W}.$ ${\cal U}^{p}\rbr{[a,b],W}$
equipped with this norm is a Banach space. \end{prop} 
\begin{proof}
For $p=1,$ $\left\Vert \cdot\right\Vert _{p-\text{TV},[a,b]}$
coincides with $V^{1}\rbr{f,[a,b]},$ $\left\Vert \cdot\right\Vert _{\text{TV},p,[a,b]}$
coincides with the $1$-variation norm $\left\Vert f\right\Vert _{\text{var},1,[a,b]}:=\left\Vert f(a)\right\Vert _{W}+V^{1}\rbr{f,[a,b]}$
and ${\cal U}^{1}\rbr{[a,b],W}$ is simply the same as the space of
functions with bounded total variation. Therefore, for the rest of
the proof we will assume that $p>1.$ 

The homogenity of $\left\Vert \cdot\right\Vert _{p-\text{TV},[a,b]}$
and $\left\Vert \cdot\right\Vert _{\text{TV},p,[a,b]}$ follows
easily from the fact that for $\alpha,\delta>0,$ $\TTV{\alpha f}{[a,b]}{\alpha\delta}=\alpha\TTV f{[a,b]}{\delta},$
which is the consequence of the equality 
\[
\max\left\{ \left\Vert \alpha f\left(t\right)-\alpha f\left(s\right)\right\Vert _{W}-\alpha\delta,0\right\} =\alpha\max\left\{ \left\Vert f\left(t\right)-f\left(s\right)\right\Vert _{W}-\delta,0\right\} .
\]

To prove the triangle inequality, let us take $f,h\in{\cal U}^{p}\rbr{[a,b]}$
and fix $\varepsilon>0.$ Let $\delta_{0}>0$ and $a\leq t_{0}<t_{1}<\ldots<t_{n}\leq b$
be such that 
\begin{align}
\left(\delta_{0}^{p-1}\sum_{i=1}^{n}\left(\left\Vert f\left(t_{i}\right)-f\left(t_{i-1}\right)+h\left(t_{i}\right)-h\left(t_{i-1}\right)\right\Vert _{W}-\delta_{0}\right)_{+}\right)^{1/p} & \geq\left\Vert f+h\right\Vert _{p-\text{TV};[a,b]}-\varepsilon,\label{eq:approx_eps}
\end{align}
where $\left(\cdot\right)_{+}$ denotes $\max\left\{ \cdot,0\right\} .$
By standard calculus, for $x>0$ and $p\geq1$ we have 
\begin{equation}
\sup_{\delta>0}\delta^{p-1}\left(x-\delta\right)_{+}=\sup_{\delta\geq0}\delta^{p-1}\left(x-\delta\right)=c_{p}x^{p}\label{eq:supremum}
\end{equation}
where $c_{p}=(p-1)^{p-1}/p^{p}\in\sbr{2^{-p};1}.$ Denote $x_{0}^{*}=0$
and for $i=1,2,\ldots,n$ define $x_{i}=\left\Vert f\left(t_{i}\right)-f\left(t_{i-1}\right)+h\left(t_{i}\right)-h\left(t_{i-1}\right)\right\Vert _{W}.$
Let $x_{1}^{*}\leq x_{2}^{*}\leq\ldots\leq x_{n}^{*}$ be the non-decreasing
re-arrangement of the sequence $\left(x_{i}\right).$ Notice that
by (\ref{eq:supremum}) for $\delta\in\left[x_{j-1}^{*};x_{j}^{*}\right],$
where $j=1,2,\ldots,n,$ one has 
\begin{align*}
 & \delta^{p-1}\sum_{i=1}^{n}\left(\left\Vert f\left(t_{i}\right)-f\left(t_{i-1}\right)+h\left(t_{i}\right)-h\left(t_{i-1}\right)\right\Vert _{W}-\delta\right)_{+}\\
 & =\delta^{p-1}\sum_{i=j}^{n}\left(x_{i}^{*}-\delta\right)=\delta^{p-1}\left(\sum_{i=j}^{n}x_{i}^{*}-\left(n-j+1\right)\delta\right)\\
 & =\left(n-j+1\right)\delta^{p-1}\left(\frac{\sum_{i=j}^{n}x_{i}^{*}}{n-j+1}-\delta\right)\leq\left(n-j+1\right)c_{p}\left(\frac{\sum_{i=j}^{n}x_{i}^{*}}{n-j+1}\right)^{p}.
\end{align*}
Hence 
\begin{align}
 & \sup_{\delta>0}\delta^{p-1}\sum_{i=1}^{n}\left(\left\Vert f\left(t_{i}\right)-f\left(t_{i-1}\right)+h\left(t_{i}\right)-h\left(t_{i-1}\right)\right\Vert _{W}-\delta\right)_{+}\nonumber \\
 & \leq\max_{j=1,2,\dots,n}\left(n-j+1\right)c_{p}\left(\frac{\sum_{i=j}^{n}x_{i}^{*}}{n-j+1}\right)^{p}.\label{eq:nier_1}
\end{align}
On the other hand, 
\begin{align}
 & \sup_{\delta>0}\delta^{p-1}\sum_{i=1}^{n}\left(\left\Vert f\left(t_{i}\right)-f\left(t_{i-1}\right)+h\left(t_{i}\right)-h\left(t_{i-1}\right)\right\Vert _{W}-\delta\right)_{+}\nonumber \\
 & =\sup_{\delta>0}\delta^{p-1}\sum_{i=1}^{n}\left(x_{i}^{*}-\delta\right)_{+}\geq\sup_{\delta>0}\max_{j=1,2,\dots,n}\delta^{p-1}\sum_{i=j}^{n}\left(x_{i}^{*}-\delta\right)\nonumber \\
 & =\max_{j=1,2,\dots,n}\sup_{\delta>0}\delta^{p-1}\sum_{i=j}^{n}\left(x_{i}^{*}-\delta\right)\nonumber \\
 & =\max_{j=1,2,\dots,n}\left(n-j+1\right)c_{p}\left(\frac{\sum_{i=j}^{n}x_{i}^{*}}{n-j+1}\right)^{p}.\label{eq:nier_2}
\end{align}
By (\ref{eq:nier_1}) and (\ref{eq:nier_2}) we get 
\begin{align}
 & \left(\sup_{\delta>0}\delta^{p-1}\sum_{i=1}^{n}\left(\left\Vert f\left(t_{i}\right)-f\left(t_{i-1}\right)+h\left(t_{i}\right)-h\left(t_{i-1}\right)\right\Vert _{W}-\delta\right)_{+}\right)^{1/p}\nonumber \\
 & =\max_{j=1,2,\dots,n}\left(n-j+1\right)^{1/p-1}c_{p}^{1/p}\sum_{i=j}^{n}x_{i}^{*}.\label{eq:optim}
\end{align}
Similarly, denoting by $y_{i}^{*}$ and $z_{i}^{*}$ the non-decreasing
rearrangements of the sequences $y_{i}=\left\Vert f\left(t_{i}\right)-f\left(t_{i-1}\right)\right\Vert _{W}$
and $z_{i}=\left\Vert h\left(t_{i}\right)-h\left(t_{i-1}\right)\right\Vert _{W}$
respectively, we get 
\begin{align*}
\left\Vert f\right\Vert _{p-\text{TV},[a,b]} & \ge\left(\sup_{\delta>0}\delta^{p-1}\sum_{i=1}^{n}\left(\left\Vert f\left(t_{i}\right)-f\left(t_{i-1}\right)\right\Vert _{W}-\delta\right)_{+}\right)^{1/p}\\
 & =\max_{j=1,2,\dots,n}\left(n-j+1\right)^{1/p-1}c_{p}^{1/p}\sum_{i=j}^{n}y_{i}^{*}
\end{align*}
and 
\begin{align*}
\left\Vert h\right\Vert _{p-\text{TV},[a,b]} & \ge\left(\sup_{\delta>0}\delta^{p-1}\sum_{i=1}^{n}\left(\left\Vert h\left(t_{i}\right)-h\left(t_{i-1}\right)\right\Vert _{W}-\delta\right)_{+}\right)^{1/p}\\
 & =\max_{j=1,2,\dots,n}\left(n-j+1\right)^{1/p-1}c_{p}^{1/p}\sum_{i=j}^{n}z_{i}^{*}.
\end{align*}
By the triangle inequality and the definition of $y_{i}^{*}$ and
$z_{i}^{*}$ for $j=1,2,\ldots,n,$ we have $\sum_{i=j}^{n}x_{i}^{*}\leq\sum_{i=j}^{n}y_{i}^{*}+\sum_{i=j}^{n}z_{i}^{*}.$
Hence 
\begin{align*}
 & \max_{j=1,2,\dots,n}\left(n-j+1\right)^{1/p-1}c_{p}^{1/p}\sum_{i=j}^{n}x_{i}^{*}\leq\max_{j=1,2,\dots,n}\left(n-j+1\right)^{1/p-1}c_{p}^{1/p}\sum_{i=j}^{n}\left(y_{i}^{*}+z_{i}^{*}\right)\\
 & \leq\max_{j=1,2,\dots,n}\left(n-j+1\right)^{1/p-1}c_{p}^{1/p}\sum_{i=j}^{n}y_{i}^{*}+\max_{j=1,2,\dots,n}\left(n-j+1\right)^{1/p-1}c_{p}^{1/p}\sum_{i=j}^{n}z_{i}^{*}\\
 & \leq\left\Vert f\right\Vert _{p-\text{TV},[a,b]}+\left\Vert h\right\Vert _{p-\text{TV},[a,b]}.
\end{align*}
Finally, by (\ref{eq:approx_eps}), (\ref{eq:optim}) and the last
estimate, we get 
\[
\left\Vert f+g\right\Vert _{p-\text{TV},[a,b]}-\varepsilon\leq\left\Vert f\right\Vert _{p-\text{TV},[a,b]}+\left\Vert g\right\Vert _{p-\text{TV},[a,b]}.
\]
Sending $\varepsilon$ to $0$ we get the triangle inequality for
$\left\Vert \cdot\right\Vert _{p-\text{TV},[a,b]}.$ From
this also follows the triangle inequality for $\left\Vert \cdot\right\Vert _{\text{TV},p,[a,b]}.$

Now we will prove that the space ${\cal U}^{p}\rbr{[a,b],W}$
equipped with the norm $\left\Vert \cdot\right\Vert _{\text{TV},p,[a,b]}$
is a Banach space. To prove this we will need the following inequality
\begin{equation}
\TTV{f+g}{[a,b]}{\delta_{1}+\delta_{2}}\leq\TTV f{[a,b]}{\delta_{1}}+\TTV g{[a,b]}{\delta_{2}}\label{TVsplit}
\end{equation}
for any $\delta_{1},\delta_{2}\geq0.$ It follows from the elementary
estimate 
\begin{equation}
\rbr{\left\Vert w_{1}-w_{2}\right\Vert _{W}-\delta_{1}-\delta_{2}}_{+}\leq\rbr{\left\Vert w_{1}\right\Vert _{W}-\delta_{1}}_{+}+\rbr{\left\Vert w_{2}\right\Vert _{W}-\delta_{2}}_{+}\label{elementary_estimate}
\end{equation}
valid for any $w_{1},w_{2}\in W$ and nonnegative $\delta_{1}$ and
$\delta_{2}.$ We also have $\TTV f{[a,b]}{\delta}\geq\rbr{\left\Vert f\right\Vert _{\text{osc},[a,b]}-\delta}_{+}.$
From this and (\ref{eq:supremum}) it follows that 
\begin{equation}
\left\Vert f\right\Vert _{\text{TV},p,[a,b]}\geq\left|f(a)\right|+c_{p}^{1/p}\left\Vert f\right\Vert _{\text{osc},[a,b]}\geq c_{p}^{1/p}\left\Vert f\right\Vert _{\ns,[a,b]},\label{infinite_norm_est}
\end{equation}
where $\left\Vert f\right\Vert _{\ns,[a,b]}:=\sup_{t\in\ab}\left\Vert f\left(t\right)\right\Vert _{W}.$
Hence any Cauchy sequence $\rbr{f_{n}}_{n=1}^{\ns}$ in ${\cal U}^{p}\rbr{[a,b],W}$
converges uniformly to some $f_{\ns}:[a,b]\ra W.$ Assume that $\left\Vert f_{\ns}-f_{n}\right\Vert _{\text{TV},p,[a,b]}\nrightarrow0$
as $n\ra+\ns.$ Thus, there exist a positive number $\kappa$, a sequence
of positive integers $n_{k}\ra+\ns$ and a sequence of positive reals
$\delta_{k},$ $k=1,2,\ldots,$ such that $\delta_{k}^{p-1}\TTV{f_{n_{k}}-f_{\ns}}{[a,b]}{\delta_{k}}\geq\kappa^{p}.$
Let $N$ be a positive integer such that 
\begin{equation}
\left\Vert f_{m}-f_{n}\right\Vert _{\text{TV},p,[a,b]}<\kappa/2^{1-1/p}\text{ for }m,n\geq N\label{Cauchy}
\end{equation}
and $k_{0}$ be the minimal positive integer such that $n_{k_{0}}\geq N.$
For sufficiently large $n\geq N$ we have $\left\Vert f_{n}-f_{\ns}\right\Vert _{\ns,[a,b]}\leq\delta_{k_{0}}/4,$
hence $\left\Vert f_{n}-f_{\ns}\right\Vert _{\text{osc},[a,b]}\leq\delta_{k_{0}}/2$
and 
\begin{equation}
\TTV{f_{n}-f_{\ns}}{[a,b]}{\delta_{k_{0}}/2}=0.\label{TVzero}
\end{equation}
Now, by (\ref{TVsplit}) 
\[
\TTV{f_{n_{k_{0}}}-f_{\ns}}{[a,b]}{\delta_{k_{0}}}\leq\TTV{f_{n_{k_{0}}}-f_{n}}{[a,b]}{\delta_{k_{0}}/2}+\TTV{f_{n}-f_{\ns}}{[a,b]}{\delta_{k_{0}}/2}.
\]
From this and (\ref{TVzero}) we get 
\[
\rbr{\delta_{k_{0}}/2}^{p-1}\TTV{f_{n_{k_{0}}}-f_{n}}{[a,b]}{\delta_{k_{0}}/2}\geq\delta_{k_{0}}^{p-1}\TTV{f_{n_{k_{0}}}-f_{\ns}}{[a,b]}{\delta_{k_{0}}}/2^{p-1}\geq\kappa^{p}/2^{p-1}
\]
but this (recall (\ref{eq:TV_norm})) contradicts (\ref{Cauchy}).
Thus, the sequence $\rbr{f_{n}}_{n=1}^{\ns}$ converges in ${\cal U}^{p}\rbr{[a,b],W}$
norm to $f_{\ns}.$ Since the sequence $\rbr{f_{n}}_{n=1}^{\ns}$
was chosen in an arbitrary way, it proves that ${\cal U}^{p}\rbr{[a,b],W}$
is complete. \end{proof} 
\begin{rem} It is easy to see that the
space ${\cal U}^{p}\rbr{[a,b],W}$ equipped with the norm $\left\Vert \cdot\right\Vert _{\emph{TV},p,[a,b]}$
is not separable. To see this it is enough for two distinct vectors
$w_{1}$ and $w_{2}$ from $W$ consider the family of functions $f_{t}:[a,b]\ra\cbr{w_{1},w_{2}},$
$f_{t}(s):={\bf 1}_{\cbr t}(s)w_{1}+(1-{\bf 1}_{\cbr t}(s))w_{2},$
$t\in\ab,$ (${\bf 1}_{A}$ denotes here the indicator function of
a set $A$) and apply (\ref{infinite_norm_est}). However, we do not
know if the subspace of continuous functions in ${\cal U}^{p}\rbr{[a,b],\R}$
is separable. 
\end{rem} 
\begin{rem} \label{superadditivity} From
the triangle inequality for $\left\Vert \cdot\right\Vert _{p-\emph{TV},[a,b]}$
it follows that it is an subadditivie functional of the interval,
i.e., for any $p\geq1,$ $f:[a,b]\ra W$ and $d\in(a,b),$ 
\[
\left\Vert f\right\Vert _{p-\emph{TV},[a,b]}\leq\left\Vert f\right\Vert _{p-\emph{TV},\left[a,d\right]}+\left\Vert f\right\Vert _{p-\emph{TV},\left[d,b\right]}.
\]
To see this it is enough to consider the following decomposition $f(t)=f_{1}(t)+f_{2}(t),$
$f_{1}(t)={\bf 1}_{[a,d]}(t)f(t)+{\bf 1}_{(d,b]}(t)f(d),$ $f_{2}(t)={\bf 1}_{(d,b]}(t)f(t)-{\bf 1}_{(d,b]}(t)f(d).$
We naturally have 
\begin{align*}
 & \left\Vert f\right\Vert _{p-\emph{TV},[a,b]}=\left\Vert f_{1}+f_{2}\right\Vert _{p-\emph{TV},[a,b]}\\
 & \leq\left\Vert f_{1}\right\Vert _{p-\emph{TV},[a,b]}+\left\Vert f_{2}\right\Vert _{p-\emph{TV},[a,b]}\\
 & =\left\Vert f\right\Vert _{p-\emph{TV},\left[a,d\right]}+\left\Vert f\right\Vert _{p-\emph{TV},\left[d,b\right]}.
\end{align*}
However, superadditivity, as a function of the interval, holding for
$\left\Vert \cdot\right\Vert _{p-\emph{var},[a,b]}^{p}=V^{p}\rbr{\cdot,[a,b]}$
is no more valid for $\left\Vert \cdot\right\Vert _{p-\emph{TV},[a,b]}^{p}.$
To see this it is enough to consider the function $f:[-1;1]\ra\cbr{-1,0,1},$
$f(t)={\bf 1}_{(-1,1)}(t)-{\bf 1}_{\cbr 1}(t).$ We have $\TTVemph f{[-1,0]}{\delta}=\rbr{1-\delta}_{+},$
$\TTVemph f{[0,1]}{\delta}=\rbr{2-\delta}_{+}$ and $\TTVemph f{[-1,1]}{\delta}=\rbr{1-\delta}_{+}+\rbr{2-\delta}_{+}$
hence $\left\Vert f\right\Vert _{2-\emph{TV},\left[-1;1\right]}^{2}=9/8<\left\Vert f\right\Vert _{2-\emph{TV},\left[-1;0\right]}^{2}+\left\Vert f\right\Vert _{2-\emph{TV},\left[0;1\right]}^{2}=1/4+1.$
\end{rem} 

\subsection{$\phi-$variation of the functions from the space ${\cal U}^{p}\left(\ab, W\right)$} 
For a (non-decreasing) function $\phi:\left[0,+\ns\right)\ra\left[0,+\ns\right)$
let us define the $\phi$-variation of $f:\ab\ra W$ as 
\[
V^{\phi}\left(f,\ab\right):=\sup_{n}\sup_{a\le t_{0}<t_{1}<\ldots<t_{n}\le b}\sum_{i=1}^{n}\phi\left(\left\Vert f\left(t_{i}\right)-f\left(t_{i-1}\right)\right\Vert _{W}\right).
\]
In this subsection we will prove
the following result. 
\begin{prop}.\label{prop_Loch} Let $p\ge1$
and suppose that $\phi:\left[0,+\ns\right)\ra\left[0,+\ns\right)$
is such that $\phi\left(0\right)=0$ and for each $t>0,$ $\phi\left(t\right)>0,$
\begin{equation}
\sup_{0<u\leq s\leq2u\leq2t}\frac{\phi\left(s\right)}{\phi\left(u\right)}<+\ns\quad\mbox{and}\quad\sum_{j=0}^{+\ns}2^{pj}\phi\left(2^{-j}\right)<+\ns.\label{eq:assumptions}
\end{equation}
Then for any function $f\in{\cal U}^{p}\left(\ab, W\right)$ one has
$\V f{\ab}{\phi}<+\ns.$ 
\end{prop} 
\begin{rem}
Function $\phi$ satifies the same assumptions as in \cite[Proposition 1]{Vovk_cadlag:2011}.
\end{rem}
\begin{proof} Let $L$ be the
least positive integer such that $\sup_{t\in\left[a,b\right]}\left\Vert f\left(t\right)\right\Vert _{W}\leq2^{L}.$
Consider the partition $\pi=\left\{ a\leq t_{0}<t_{1}<\ldots<t_{n}\leq b\right\} $ such that $f\left(t_{i}\right)\neq f\left(t_{i-1}\right)$ for $i=1,2,\ldots,n$
and for $j=0,1,\ldots$ define 
\[
I_{j}=\left\{ i\in\left\{ 1,2,\ldots,n\right\} :\left\Vert f\left(t_{i}\right)-f\left(t_{i-1}\right)\right\Vert _{W}\in\left(2^{L-j},2^{L-j+1}\right]\right\} ,
\]
and $\delta\left(j\right):=2^{L-j-1}.$ Naturally, for $i\in I_{j},$
\[
\left\Vert f\left(t_{i}\right)-f\left(t_{i-1}\right)\right\Vert _{W}-\delta\left(j\right)\geq\frac{1}{2}\left\Vert f\left(t_{i}\right)-f\left(t_{i-1}\right)\right\Vert _{W}
\]
and since $\left\{ 1,2,\ldots,n\right\} =\bigcup_{j=0}^{+\infty}I_{j},$
we estimate 
\begin{align}
 & \sum_{i=1}^{n}\phi\left(\left\Vert f\left(t_{i}\right)-f\left(t_{i-1}\right)\right\Vert _{W}\right)\nonumber  = \sum_{j=0}^{+\infty}\sum_{i\in I_{j}}\phi\left(\left\Vert f\left(t_{i}\right)-f\left(t_{i-1}\right)\right\Vert _{W}\right)\nonumber \\
 & \leq\sum_{j=0}^{+\infty}\sup_{s\in\left[2^{L-j},2^{L-j+1}\right]}\frac{\phi\left(s\right)}{2^{L-j}}\sum_{i\in I_{j}}\left\Vert f\left(t_{i}\right)-f\left(t_{i-1}\right)\right\Vert _{W}\nonumber \\
 & \leq\sum_{j=0}^{+\infty}\sup_{s\in\left[2^{L-j},2^{L-j+1}\right]}\frac{\phi\left(s\right)}{2^{L-j}}\cdot2\sum_{i\in I_{j}}\max\left\{ \left\Vert f\left(t_{i}\right)-f\left(t_{i-1}\right)\right\Vert _{W}-\delta\left(j\right),0\right\} \nonumber \\
 & \leq\sum_{j=0}^{+\infty}\sup_{s\in\left[2^{L-j},2^{L-j+1}\right]}\frac{\phi\left(s\right)}{2^{L-j}}\cdot2\TTV f{[a,b]}{\delta\left(j\right)}\\
 & =\sum_{j=0}^{+\infty}\sup_{s\in\left[2^{L-j},2^{L-j+1}\right]}\frac{\phi\left(s\right)}{2^{L-j}}\cdot2\cdot\frac{1}{\delta\left(j\right)^{p-1}}\delta\left(j\right)^{p-1}\TTV f{[a,b]}{\delta\left(j\right)}\nonumber \\
 & \leq\sum_{j=0}^{+\infty}\sup_{s\in\left[2^{L-j},2^{L-j+1}\right]}\frac{\phi\left(s\right)}{2^{L-j}}\cdot2\cdot\frac{1}{\delta\left(j\right)^{p-1}}\sup_{\delta>0}\left\{ {\delta}^{p-1}\cdot\TTV f{[a,b]}{\delta}\right\} \\
 & =2^p\left\Vert f\right\Vert _{p-\TV,\ab}^{p}\sum_{j=0}^{+\infty}2^{p\left(j-L\right)}\sup_{s\in\left[2^{L-j},2^{L-j+1}\right]}\phi\left(s\right).\label{estim_last1}
\end{align}
By the first assumption in (\ref{eq:assumptions}) we have that for
all $j=0,1,\ldots,$ 
\[
\sup_{s\in\left[2^{L-j},2^{L-j+1}\right]}\phi\left(s\right)\leq C\left(\phi,L\right)\cdot\phi\left(2^{L-j}\right)
\]
for some constant $C\left(\phi,L\right)$ depending on $\phi$ and
$L$ only. Thus, by the second assumption in (\ref{eq:assumptions}),
$\sum_{j=0}^{+\ns}2^{pj}\phi\left(2^{-j}\right)<+\ns,$ we get 
\begin{align}
 & \sum_{j=0}^{+\infty}2^{p\left(j-L\right)}\sup_{s\in\left[2^{L-j},2^{L-j+1}\right]}\phi\left(s\right)\nonumber \\
 & \leq C\left(\phi,L\right)\left\{ \sum_{j=0}^{L-1}2^{p\left(j-L\right)}\phi\left(2^{L-j}\right)+\sum_{j=L}^{+\infty}2^{p\left(j-L\right)}\phi\left(2^{L-j}\right)\right\} \nonumber \\
 & =C\left(\phi,L\right)\left\{ \sum_{j=0}^{L-1}2^{p\left(j-L\right)}\phi\left(2^{L-j}\right)+\sum_{j=0}^{+\ns}2^{pj}\phi\left(2^{-j}\right)\right\} <+\ns.\label{estim_last2}
\end{align}
Since estimates (\ref{estim_last1}) and (\ref{estim_last2}) do not
depend on the partition $\pi,$ taking supremum over all partitions
of the interval $\ab$ we get $\V f{\ab}{\phi}<+\ns.$ 
\end{proof} 
\begin{rem}
From Proposition \ref{prop_Loch} it immediately follows that ${\cal U}^{p}\left(\ab,W\right)\subset{\cal V}^{q}\left(\ab,W\right)$
for any $q>p,$ since for any $q>p,$ $\phi_{q}\left(x\right)=x^{q}$
satisfies (\ref{eq:assumptions}). But it is easy to derive more exact
results. For example, by standard calculus, assumptions (\ref{eq:assumptions}) hold for 
\[\phi_{p,\gamma,1}\left(x\right):=\frac{x^{p}}{\rbr{\ln\left(1+1/x\right)}^{\gamma}} \text{ or } \phi_{p,\gamma,2}\left(x\right):=\frac{x^{p}}{\ln\left(1+1/x\right)\left(\ln\ln\left(e+1/x\right)\right)^{\gamma}} \]
when $\gamma>1.$ From this we have that $\bigcap_{q>p}{\cal V}^{q}\left(\ab,W\right)\neq{\cal U}^{q}\left(\ab,W\right)$
since there exist functions $f:\ab\ra W$ such that $\V f{\ab}{\phi_{p,2,1}}=+\ns$
but $f\in{\cal V}^{q}\left(\ab, W\right)$ for any $q>p.$ An example of such a function is the following. Let ${ w} \in W$ be such that $\left\Vert { w} \right\Vert_W = 1$ then $f:[0,1] \ra W$ is defined as 
\[
f(t)=\begin{cases}\rbr{{\ln n}/{n}}^{1/p} { w}
& \text{if }  t = {1}/{n} \text{ for } n=1,2,\ldots ;\\
{0 } & \text{otherwise.}
\end{cases}
\]
It remains an open question if it is possible to obtain finiteness of the $\phi-$variation of functions from ${\cal U}^{p}\left(\ab, W\right)$ for $\phi$ vanishing slower (as $x \ra 0+$) than $\frac{x^{p}}{{\ln\left(1+1/x\right)}} .$
\end{rem}

\subsection{Irregularity of the integrals driven by functions from ${\cal U}^{p}\left([a,b],W\right)$}

In \cite[Section 2]{Lyons:1994} there are considered $\left\Vert \cdot\right\Vert _{p-\text{var},\left[a,b\right]}$
norms of the integrals of the form $[a,b]\ni t\mapsto\int_{a}^{t}f\dd g,$
with $f\in{\cal V}^{p}\rbr{[a,b],L\left(E,V\right)}$ and $g\in{\cal V}^{q}\rbr{[a,b],E},$
where $p>1,$ $q>1$ and $p^{-1}+q^{-1}>1.$ Now, we turn to investigate
the $\left\Vert \cdot\right\Vert _{p-\text{TV},\left[a,b\right]}$
norms of similar integrals, but for $f\in{\cal U}^{p}\rbr{[a,b],L\left(E,V\right)}$
and $g\in{\cal U}^{q}\rbr{[a,b],W}.$ 
This, in view of the preceding subsection, will give us more exact results about the irregularity of the indefinite integrals $\int_{a}^{\cdot}f\dd g.$ 
We will prove the following
\begin{theo} \label{thm_integral}
Assume that $f\in{\cal U}^{p}\rbr{[a,b],L\left(E,V\right)}$ and $g\in{\cal U}^{q}\rbr{[a,b],E}$
for some $p>1,$ $q>1,$ such that $p^{-1}+q^{-1}>1$ and they have
no common points of discontinuity. Then there exist a constant $D_{p,q}<+\infty,$
depending on $p$ and $q$ only, such that 
\[
\left\Vert \int_{a}^{\cdot}\left[f\left(s\right)-f\left(a\right)\right]\mathrm{d}g\left(s\right)\right\Vert _{q-\emph{TV},\left[a,b\right]}\leq D_{p,q}\left\Vert f\right\Vert _{p-\emph{TV},\left[a,b\right]}^{p-p/q}\left\Vert f\right\Vert _{\emph{osc},\left[a,b\right]}^{1+p/q-p}\left\Vert g\right\Vert _{q-\emph{TV},\left[a,b\right]}.
\]
\end{theo} 

One more application of Theorem \ref{thm_integral} may be the following. Assume that $y:\left[a, b \right] \ra E$ is a solution of the equation of the form 
\begin{equation} \label{int_eq}
y(t) = x_0 + \int_a^t F\rbr{s,y(s)} \dd x(s), 
\end{equation}
where $x\in {\cal U}^{q}\rbr{[a,b],E}$ and $F\rbr{\cdot,y(\cdot)} \in {\cal U}^{p}\rbr{[a,b],L(E,V)}$ for some $p>1,$ $q>1$ such that $p^{-1}+q^{-1} >1;$ from these and Theorem \ref{thm_integral}  we will obtain that $y \in {\cal U}^{q}\rbr{[a,b],E}.$  

In our case we have no longer the supperadditivity property
of the functional $\left\Vert \cdot\right\Vert _{p-\text{TV},\left[a,b\right]}^{p}$
as the function of interval (see Remark \ref{superadditivity}), hence
the method of the proof of Theorem \ref{thm_integral} will be different
than the proofs of related estimates in \cite{Lyons:1994}. It will
be similar to the proof of Corollary \ref{corol_Young}. We will need
the following lemma. 
\begin{lema} \label{lema1} Let $f:\ab\ra L\left(E,V\right)$
and $g:\ab\ra E,$ be two regulated functions which have no common
points of discontinuity and $\delta_{0}\geq\delta_{1}\geq\ldots,$
$\varepsilon_{0}\geq\varepsilon_{1}\geq\ldots$ be two sequences of
non-negative numbers, such that $\delta_{k}\downarrow0,$ $\varepsilon_{k}\downarrow0$
as $k\rightarrow+\infty.$ Assume that for $\delta_{-1}:=\frac{1}{2}\sup_{a\leq t\leq b}\left\Vert f\left(t\right)-f\left(a\right)\right\Vert_{L(E,V)}$
and 
\[
S=4\sum_{k=0}^{+\infty}3^{k}\delta_{k-1}\cdot\TTVemph g{\left[a,b\right]}{\varepsilon_{k}}+4\sum_{k=0}^{\infty}3^{k}\varepsilon_{k}\cdot\TTVemph f{\left[a,b\right]}{\delta_{k}}
\]
we have $S<+\ns.$ Defining 
\[
\gamma:=8\sum_{k=0}^{+\infty}3^{k}\varepsilon_{k}\cdot\TTVemph f{\left[a,b\right]}{\delta_{k}}
\]
we get 
\begin{align*}
\TTVemph{\int_{a}^{\cdot}\left[f\left(s\right)-f\left(a\right)\right]\mathrm{d}g\left(s\right)}{\left[a,b\right]}{\gamma} & \leq 2\sum_{k=0}^{+\infty}3^{k}\delta_{k-1}\cdot\TTVemph g{\left[a,b\right]}{\varepsilon_{k}}.
\end{align*}
\end{lema} \begin{proof} We proceed similarly as in the proof of
Lemma \ref{lema_partition}. Define $g_{0}=g,$ $f_{0}=f,$ $g_{1}:=g_{0}-g_{0}^{\varepsilon_{0}},$
$f_{1}:=f_{0}-f_{0}^{\delta_{0}},$ where $g_{0}^{\varepsilon_{0}}$
is piecewise linear, with possible discontinuities only at the points
where $g$ is discontinuous, and such that 
\[
\left\Vert g_{0}-g_{0}^{\varepsilon_{0}}\right\Vert _{\infty,\ab}\leq\varepsilon_{0}\mbox{ and }\TTV{g_{0}^{\varepsilon_{0}}}{\left[a,b\right]}0\le2\TTV{g_{0}}{\left[a,b\right]}{\varepsilon_{0}/4},
\]
and, similarly, $f_{0}^{\delta_{0}}$ is piecewise linear, with possible discontinuities
only at the points where $f$ is discontinuous, and such that 
\[
\left\Vert f_{0}-f_{0}^{\delta_{0}}\right\Vert _{\infty,\ab}\leq\delta_{0}\mbox{ and }\TTV{f_{0}^{\delta_{0}}}{\left[a,b\right]}0\le2\TTV{f_{0}}{\left[a,b\right]}{\delta_{0}/4}.
\]
For $k=2,3,\ldots,$ $g_{k}:=g_{k-1}-g_{k-1}^{\varepsilon_{k-1}},$
$f_{k}:=f_{k-1}-f_{k-1}^{\delta_{k-1}}$ are defined similarly as
$g_{1}$ and $f_{1}.$ By the linearity of the Riemann-Stieltjes integral
with respect to the integrator, integrating by parts, for $t\in\left[a,b\right],$
$r=1,2,\ldots,$ we have 
\begin{align}
& \int_{a}^{t}\left[f\left(s\right)-f\left(a\right)\right]\mathrm{d}g\left(s\right) \nonumber \\ & =\int_{a}^{t}\left[f_{0}\left(s\right)-f_{0}\left(a\right)\right]\mathrm{d}g^{\varepsilon_{0}}\left(s\right)+\int_{a}^{t}\left[f_{0}\left(s\right)-f_{0}\left(a\right)\right]\mathrm{d}g_{1}\left(s\right)\nonumber \\
 & =\int_{a}^{t}\left[f_{0}\left(s\right)-f_{0}\left(a\right)\right]\mathrm{d}g_{0}^{\varepsilon_{0}}\left(s\right)+\int_{a}^{t}\mathrm{d}f_{0}\left(s\right)\left[g_{1}\left(t\right)-g_{1}\left(s\right)\right]\nonumber \\
 & =\int_{a}^{t}\left[f_{0}\left(s\right)-f_{0}\left(a\right)\right]\mathrm{d}g_{0}^{\varepsilon_{0}}\left(s\right)+\int_{a}^{t}\mathrm{d}f_{0}^{\delta_{0}}\left(s\right)\left[g_{1}\left(t\right)-g_{1}\left(s\right)\right]\nonumber \\
 & \quad+\int_{a}^{t}\mathrm{d}f_{1}\left(s\right)\left[g_{1}\left(t\right)-g_{1}\left(s\right)\right]\nonumber \\
 & =\int_{a}^{t}\left[f_{0}\left(s\right)-f_{0}\left(a\right)\right]\mathrm{d}g_{0}^{\varepsilon_{0}}\left(s\right)+\int_{a}^{t}\mathrm{d}f_{0}^{\delta_{0}}\left(s\right)\left[g_{1}\left(t\right)-g_{1}\left(s\right)\right]\nonumber \\
 & \quad+\int_{a}^{t}\left[f_{1}\left(s\right)-f_{1}\left(a\right)\right]\mathrm{d}g_{1}\left(s\right)=\ldots\nonumber \\
 & =\sum_{k=0}^{r-1}\left(\int_{a}^{t}\left[f_{k}\left(s\right)-f_{k}\left(a\right)\right]\mathrm{d}g_{k}^{\varepsilon_{k}}\left(s\right)+\int_{a}^{t}\mathrm{d}f_{k}^{\delta_{k}}\left(s\right)\left[g_{k+1}\left(t\right)-g_{k+1}\left(s\right)\right]\right)\nonumber \\
 & \quad+\int_{a}^{t}\left[f_{r}\left(s\right)-f_{r}\left(a\right)\right]\mathrm{d}g_{r}\left(s\right).\label{eq:summaa}
\end{align}
By Theorem \ref{main}, we easily estimate that 
\begin{align}
& \left\Vert \int_{a}^{t}\left[f_{r}\left(s\right)-f_{r}\left(a\right)\right]\mathrm{d}g_{r}\left(s\right)\right\Vert _{V}\nonumber  \\ 
& \leq4 \sum_{k=r}^{+\infty}3^{k}\delta_{k-1}\cdot\TTV g{\left[a,t\right]}{\varepsilon_{k}/4}+4\sum_{k=r}^{\infty}3^{k}\varepsilon_{k}\cdot\TTV f{\left[a,t\right]}{\delta_{k}/4}\label{eq:tv_gamma_est1}
\end{align}
for $r=1,2,\ldots.$ Moreover, for $k=0,1,\ldots,$ similarly as in
the proof of Lemma \ref{lema_partition}, we estimate 
\begin{equation}
\left\Vert \int_{a}^{t}\mathrm{d}f_{k}^{\delta_{k}}\left(s\right)\left[g_{k+1}\left(t\right)-g_{k+1}\left(s\right)\right]\right\Vert _{V}\leq2\varepsilon_{k}\TTV{f_{k}^{\delta_{k}}}{\left[a,t\right]}0\leq2\cdot3^{k}\varepsilon_{k}\TTV f{\left[a,t\right]}{\delta_{k}/4},\label{eq:tv_gamma_est2}
\end{equation}
and 
\begin{align}
\TTV{\int_{a}^{\cdot}\left[f_{k}\left(s\right)-f_{k}\left(a\right)\right]\mathrm{d}g_{k}^{\varepsilon_{k}}\left(s\right)}{\left[a,b\right]}0 & \leq2\delta_{k-1}\TTV{g_{k}^{\varepsilon_{k}}}{\left[a,b\right]}0\nonumber \\
 & \leq2\cdot3^{k}\delta_{k-1}\cdot\TTV g{\left[a,b\right]}{\varepsilon_{k}/4}.\label{eq:tv_gamma_est3}
\end{align}
(Notice that for the function $F_{k}(t):=\int_{a}^{t}\left[g_{k+1}\left(t\right)-g_{k+1}\left(s\right)\right]\mathrm{d}f_{k}^{\delta_{k}}\left(s\right)$
we could not obtain a similar estimate as (\ref{eq:tv_gamma_est3}).
This is due to the fact that $F_{k}(t_{2})-F_{k}(t_{1})$ can not
be expressed as the integral $\int_{t_{1}}^{t_{2}}\left[g_{k+1}\left(t\right)-g_{k+1}\left(s\right)\right]\mathrm{d}f_{k}^{\delta_{k}}\left(s\right)$
.) Defining 
\[
\gamma\left(r\right):=8\sum_{k=r}^{+\infty}3^{k}\delta_{k-1}\cdot\TTV g{\left[a,b\right]}{\varepsilon_{k}}+8\sum_{k=0}^{+\ns}3^{k}\varepsilon_{k}\cdot\TTV f{\left[a,b\right]}{\delta_{k}/4},
\]
from (\ref{eq:summaa}), (\ref{eq:tv_gamma_est2}) and (\ref{eq:tv_gamma_est1})
we get 
\begin{align*}
 & \left\Vert \int_{a}^{t}\left[f\left(s\right)-f\left(a\right)\right]\mathrm{d}g\left(s\right)-\sum_{k=0}^{r-1}\int_{a}^{t}\left[f_{k}\left(s\right)-f_{k}\left(a\right)\right]\mathrm{d}g_{k}^{\varepsilon_{k}}\left(s\right)\right\Vert _{V}\\
 & \leq\sum_{k=0}^{r-1}\left\Vert \int_{a}^{t}\mathrm{d}f_{k}^{\delta_{k}}\left(s\right)\left[g_{k+1}\left(t\right)-g_{k+1}\left(s\right)\right]\right\Vert _{V}+\left\Vert \int_{a}^{t}\left[f_{r}\left(s\right)-f_{r}\left(a\right)\right]\mathrm{d}g_{r}\left(s\right)\right\Vert _{V}\\
 & \leq\frac{1}{2}\gamma\left(r\right)
\end{align*}
for any $t\in\left[a,b\right].$ Let us notice that by the very definition
of the truncated variation, 
\[
\TTV{\int_{a}^{\cdot}\left[f\left(s\right)-f\left(a\right)\right]\mathrm{d}g\left(s\right)}{\left[a,t\right]}{\gamma}
\]
is bouned from above by the variation of any function approximating the indefinite integral 
$\int_{a}^{\cdot}\left[f\left(s\right)-f\left(a\right)\right]\mathrm{d}g\left(s\right)$
with accuracy $\gamma/2.$ By this variational property of the truncated
variation and by (\ref{eq:tv_gamma_est3}) we get 
\begin{align*}
& \TTV{\int_{a}^{\cdot}\left[f\left(s\right)-f\left(a\right)\right]\mathrm{d}g\left(s\right)}{\left[a,b\right]}{\gamma\left(r\right)} \\ & \quad \quad \quad  \leq \TTV{\sum_{k=0}^{r-1}\int_{a}^{\cdot}\left[f_{k}\left(s\right)-f_{k}\left(a\right)\right]\mathrm{d}g_{k}^{\varepsilon_{k}}\left(s\right)}{\left[a,b\right]}0\\
& \quad \quad \quad \leq  2\sum_{k=0}^{r-1}3^{k}\delta_{k-1}\cdot\TTV g{\left[a,b\right]}{\varepsilon_{k}/4}.
\end{align*}
Proceeding with $r$ to $+\infty$ we get the assertion. \end{proof} 

Now we are ready to prove Theorem \ref{thm_integral}. 

\begin{proof} Let $\gamma>0.$ We choose 
\[ \alpha=\frac{\sqrt{(q-1)(p-1)}+1}{2}, \quad \delta_{-1}=\frac{1}{2}\sup_{a\leq t\leq b}\left\Vert f\left(t\right)-f\left(a\right)\right\Vert _{L\left(E,V\right)}
\]
and define $\beta$ by the equality
\[
2\cdot4^{p}\left(\sum_{k=0}^{+\infty}3^{k+1-p-\alpha\left(1-\alpha\right)\left(\alpha^{2}/\left[\left(q-1\right)\left(p-1\right)\right]\right)^{k}/\left(q-1\right)}\right)\left\Vert f\right\Vert _{p-\text{TV},\left[a,b\right]}^{p}\delta_{-1}^{1-p}\beta=\gamma.
\]
Now, for $k=0,1,\ldots,$ we define 
\[
\delta_{k-1}=3^{-\left(\alpha^{2}/\left[\left(q-1\right)\left(p-1\right)\right]\right)^{k}+1}\delta_{-1},
\]
\[
\varepsilon_{k}=3^{-\left(\alpha^{2}/\left[\left(q-1\right)\left(p-1\right)\right]\right)^{k}\alpha/\left(q-1\right)}\beta.
\]
Using (\ref{TV_estimate}), similarly as in the proof of Corollary
\ref{corol_Young} { we estimate 
\begin{align*}
& \sum_{k=0}^{+\infty}3^{k}\delta_{k-1}\cdot\TTV g{\left[a,b\right]}{\varepsilon_{k}/4} \\ & \leq4^{q-1}\left(\sum_{k=0}^{+\infty}3^{k+1-\left(1-\alpha\right)\left(\alpha^{2}/\left[\left(q-1\right)\left(p-1\right)\right]\right)^{k}}\right)\left\Vert g\right\Vert _{q-\text{TV},\left[a,b\right]}^{q}\delta_{-1}\beta^{1-q}
\end{align*}
and 
\begin{align*}
\tilde{\gamma}: & =8\sum_{k=0}^{+\infty}3^{k}\varepsilon_{k}\cdot\TTV f{\left[a,b\right]}{\delta_{k}/4}\\
 & \leq2\cdot4^{p}\left(\sum_{k=0}^{+\infty}3^{k+1-p-\alpha\left(1-\alpha\right)\left(\alpha^{2}/\left[\left(q-1\right)\left(p-1\right)\right]\right)^{k}/\left(q-1\right)}\right)\left\Vert f\right\Vert _{p-\text{TV},\left[a,b\right]}^{p}\delta_{-1}^{1-p}\beta\\
 & =\gamma.
\end{align*}
By the monotonicity of the truncated variation, Lemma \ref{lema1}
and the last two estimates we get 
\begin{align*}
 & \TTV{\int_{a}^{\cdot}\left[f\left(s\right)-f\left(a\right)\right]\mathrm{d}g\left(s\right)}{\left[a,b\right]}{\gamma}\leq\TTV{\int_{a}^{\cdot}\left[f\left(s\right)-f\left(a\right)\right]\mathrm{d}g\left(s\right)}{\left[a,b\right]}{\tilde{\gamma}}\\
 & \leq4\sum_{k=0}^{+\infty}3^{k}\delta_{k-1}\cdot\TTV g{\left[a,b\right]}{\varepsilon_{k}}\\
 & \leq4^{q}\left(\sum_{k=0}^{+\infty}3^{k+1-\left(1-\alpha\right)\left(\alpha^{2}/\left[\left(q-1\right)\left(p-1\right)\right]\right)^{k}}\right)\left\Vert g\right\Vert _{q-\text{TV},\left[a,b\right]}^{q}\delta_{-1}\beta^{1-q}\\
 & =\tilde{D}_{p,q}\left\Vert f\right\Vert _{p-\text{TV},\left[a,b\right]}^{pq-p}\left\Vert f\right\Vert _{\text{osc},\left[a,b\right]}^{p+q-pq}\left\Vert g\right\Vert _{q-\text{TV},\left[a,b\right]}^{q}\gamma^{1-q},
\end{align*}
where 
\begin{align*}
\tilde{D}_{p,q}= & 4^{q}\left(\sum_{k=0}^{+\infty}3^{k+1-\left(1-\alpha\right)\left(\alpha^{2}/\left[\left(q-1\right)\left(p-1\right)\right]\right)^{k}}\right)\\
 & \times\left(2\cdot4^{p}\left(\sum_{k=0}^{+\infty}3^{k+1-p-\alpha\left(1-\alpha\right)\left(\alpha^{2}/\left[\left(q-1\right)\left(p-1\right)\right]\right)^{k}/\left(q-1\right)}\right)\right)^{q-1}.
\end{align*}
} From this and the definition of $\left\Vert \cdot\right\Vert _{q-\text{TV},\left[a,b\right]}$
we get
\[
\left\Vert \int_{a}^{\cdot}\left[f\left(s\right)-f\left(a\right)\right]\mathrm{d}g\left(s\right)\right\Vert _{q-\text{TV},\left[a,b\right]}\leq D_{p,q}\left\Vert f\right\Vert _{p-\text{TV},\left[a,b\right]}^{p-p/q}\left\Vert f\right\Vert _{\text{osc},\left[a,b\right]}^{p/q+1-p}\left\Vert g\right\Vert _{q-\text{TV},\left[a,b\right]},
\]
where $D_{p,q}=\tilde{D}_{p,q}^{1/q}.$ \end{proof}

\bibliographystyle{plain}
\bibliography{C:/biblio/biblio}
\end{document}